\documentclass[a4paper,11pt]{amsart}
\usepackage[utf8]{inputenc}
\usepackage{amsmath, amssymb}
\usepackage{setspace}
\usepackage[left=3.3cm, right=3.3cm, top=3cm, bottom=3cm]{geometry}                 

\usepackage{graphicx} 
\usepackage{pgf,tikz}
\usetikzlibrary{arrows}
\usepackage{amssymb}
\usepackage{pdfsync}
\usepackage{mathrsfs}
\usepackage{hyperref} 
\usepackage{verbatim} 
\usepackage{epstopdf}
\usepackage{xargs}
\usepackage{todonotes}
\usepackage[normalem]{ulem}

\newcommandx{\jow}[2][1=]{\todo[linecolor=orange,backgroundcolor=orange!25,bordercolor=orange,#1]{#2}}

\newcommandx{\mateus}[2][1=]{\todo[linecolor=blue,backgroundcolor=blue!25,bordercolor=blue,#1]{#2}}

\def\R{\mathbb{R}}

\def\N{\mathbb{N}}
\def\Z{\mathbb{Z}}

\def\C{\mathbb{C}}

\def\ee{\mathrm{e}}
\def\HH{\mathbb{H}}
\def\ff{\mathfrak{f}}

\newcommand{\eps}{\varepsilon}
\newcommand{\mmd}{\mathrm{d}}

\renewcommand{\H}{\mathcal{H}}

\newcommandx{\eq}{\approxeq}
\newcommand{\supp}{\mathrm{supp}}

\renewcommand{\H}{\mathcal{H}}

\newcommand{\sinc}{\mathrm{sinc}}

\newtheorem{theorem}{Theorem}
\newtheorem{corollary}[theorem]{Corollary}

\newtheorem{proposition}[theorem]{Proposition}
\newtheorem{lemma}[theorem]{Lemma}

\newtheorem{claim}[theorem]{Claim}
\newtheorem{question}{Question}

\newtheorem{conjecture}[theorem]{Conjecture}

\numberwithin{equation}{section}
\numberwithin{theorem}{section}

\begin{document}
\title{Perturbed interpolation formulae and applications}
\author{Jo\~ao P. G. Ramos and Mateus Sousa}
\maketitle

\begin{abstract}
We employ functional analysis techniques in order to deduce that some classical and recent interpolation results in Fourier analysis can be suitably perturbed. As an application of our techniques, we obtain generalizations 
of Kadec's $1/4-$theorem for interpolation formulae in the Paley--Wiener space both in the real and complex case, as well as a perturbation result on the recent Radchenko--Viazovska interpolation result \cite{RV} and the 
Cohn--Kumar--Miller--Radchenko--Viazovska \cite{CKMRV2} result for Fourier interpolation with derivatives in dimensions $8$ and $24.$ We also provide several applications of the main results and techniques, all relating to 
recent contributions in interpolation formulae and uniqueness sets for the Fourier transform. 
\end{abstract}

\tableofcontents

\section{Introduction}

A fundamental question in analysis is that of how to recover a function $f$ from some subset $\{f(x)\}_{x \in A}$ of its values, together with some information on its 
\emph{Fourier transform} $\widehat{f}:\R \to \C$, which we define to be
\begin{equation}\label{eq Fourier transform}
\widehat{f}(\xi) = \int_{\R} f(x) e^{-2\pi i x \xi} \, \mmd x.
\end{equation} 
The perhaps most classical result in that regard is the \emph{Shannon--Whittaker interpolation formula}: if $\widehat{f}$ is supported on an interval $[-\delta/2,\delta/2],$ then 
\begin{equation}\label{eq shannon form}
    f(x)=\sum_{k=-\infty}^{\infty}f(k/\delta)\sinc(\delta x- k),
\end{equation}
where convergence holds both in $L^2(\R)$ and uniformly, where we let $\sinc(x)=\tfrac{\sin(\pi x)}{\pi x}$. In spite of this classical formula, a major recent breakthrough in regard to the problem of determining which conditions
on the sets $A, B \subset \R$ imply that a function $f \in \mathcal{S}(\R)$ is uniquely determined by its values at $A$ and the values of its Fourier transform at $B$ was made by Radchenko and Viazovska \cite{RV}, where the authors prove that, whenever $f: \R \to \R$ is even and Schwartz, 
then 
\begin{equation}\label{eq:interpolation_schwartz}
    f(x)=\sum_{k=0}^{\infty}f(\sqrt{k})a_k(x)+\sum_{k=0}^{\infty}\widehat{f}(\sqrt{k})\widehat{a_k}(x).
\end{equation}
Radchenko and Viazovska's result and its techniques were somewhat inspired by Viazovska's recent solution to the sphere packing problem in dimension 8 \cite{Viazovska1}, and her subsequent work with Cohn, Kumar, Miller and Radchenko to solve 
the same problem in dimension 24 \cite{CKMRV1}, as they include the usage of \emph{modular forms} in order to construct some special functions with particular properties at the desired nodes of interpolation. 

Subsequently to the Radchenko--Viazovska result, other recent works have successfully used a similar approach in order to construct interpolation and uniqueness formulae. Among those, we mention the following: 
\begin{enumerate}
 \item In \cite{CohnGoncalves}, Cohn and Gon\c calves use a modular form construction in order to obtain that there are $c_j > 0, \, j \in \N,$ so that, for each $f \in \mathcal{S}_{rad}(\R^{12})$ real, 
 \begin{equation}\label{eq poisson felipe}
 f(0) - \sum_{j \ge 1} c_j f(\sqrt{2j}) = - \widehat{f}(0) + \sum_{j \ge 1} c_j \widehat{f}(\sqrt{2j}).
 \end{equation}
 Such a formula enables the authors to prove a sharp version of a root uncertainty principle first raised by Bourgain, Clozel and Kahane \cite{BCK} in dimension 12; see, e.g., \cite{GOSS17, GOSR20, GOSR201} and the references therein for more information 
 on this topic;
 \item On the other hand, in \cite{CKMRV2}, Cohn, Kumar, Miller, Radchenko and Viazovska develop upon the basic ideas of \cite{RV} to be able to prove \emph{universal optimality results} about the $E_8$ and Leech lattices in dimensions 
 $8$ and $24,$ respectively. In order to do so, they prove interpolation formulae in such dimensions that involve the values of $f(\sqrt{2n}), f'(\sqrt{2n}), \widehat{f}(\sqrt{2n}), \widehat{f}'(\sqrt{2n}),$ where $f$ is a radial, Schwartz function, 
 and $n \ge n_0,$ with $n_0 = 1$ if $d = 8,$ and $n_0 = 2$ in case $d=24;$ 
 \item Finally, more recently, other developments in the theory of interpolation formulae given values on both Fourier and spatial side has been made by Stoller \cite{Stoller}, who considered the problem of recovering any funtion in $\R^d$ from 
 its restrictions and the restrictions of its Fourier transforms to spheres of radii $\sqrt{n}, \, n >0,$ for any $d > 0.$ Moreover, we mention also the more recent work of Bondarendo, Radchenko and Seip \cite{BRS}, which generalizes Radchenko and Viazovska's 
 construction of the interpolating functions to prove interpolation formulae for some classes of functions $f$ that take into account the values of $\widehat{f}$ at $\log{n}/4\pi,$ and the values of $f$ at a sequence $(\rho-1/2)/i,$ where $\rho$ ranges over non-trivial zeros of some $L-$function 
 with positive imaginary part. 
\end{enumerate}

One fundamental point to stress is that, in a suitable way, all the previously mentioned results are related to some sort of \emph{summation formula}, the most basic instance of such being the classical Poisson summation formula 
\begin{equation*}
\sum_{m \in \Z} f(m) = \sum_{n \in \Z} \widehat{f}(n),
\end{equation*}
which is a particular case, for instance, of \eqref{eq:interpolation_schwartz} in case we set $x =0.$ Clearly, the formula \eqref{eq poisson felipe} is also a manifestation of such a principle that implies rigidity between certain
values of $f$ and other values of $\widehat{f}$.

In that regard, these topics can be inserted into the framework of \emph{crystalline measures}. Indeed, if we adopt the classical definition of 
a crystalline measure to be a distribution with locally finite support, such that its Fourier transform possesses the same support property, we will see that the Poisson summation formula implies, for instance, that 
the measure $\delta_{\Z}$ is not only a crystalline measure, but also \emph{self-dual}, in the sense that $\delta_{\Z} = \widehat{\delta}_{\Z}$ holds in $\mathcal{S}'(\R).$ 

Outside the scope of interpolation formulae per se, we mention the works \cite{LO13, LO15, Meyer}, where the authors explore in a deeper lever structural questions on crystalline measures. 
In particular, in \cite{Meyer}, Meyer exhibits examples of crystalline measures with self-duality properties, and uses modular forms 
to construct explicity examples of non-zero self-dual crystalline measures $\mu$ supported on $\{\pm\sqrt{k+a}, k \in \Z\},$  for $a \in \{9,24,72\}.$ We also mention the recent work of Kurasov and Sarnak \cite{KurasovSarnak}, 
where the authors, as a by-product of investigations of the additive structure of the spectrum of metric graphs, prove that there are exotic examples of \emph{positive} crystalline measures other than generalized Dirac combs. 

Our investigation in this paper focuses on both classical and modern results in the theory of such interpolation formulae and crystalline measures. In generic terms, we are interested in determining when, given an interpolation
formula such as \eqref{eq shannon form} or \eqref{eq:interpolation_schwartz}, we can \emph{perturb} it suitably. That is, given a sequence of real numbers $\{\eps_k\}_{k \in \Z},$ under which conditions can we recover $f$ from the values 
\begin{equation}\label{eq perturbed}
\{ (f(s_n+\eps_n), \widehat{f}(\widehat{s_n}+\eps_n))\}_{n \in \Z},
\end{equation}
given that we can recover $f$ from $\{(f(s_n),\widehat{f}(\widehat{s_n}))\}_{n \in \Z}$? 

In this manuscript, the main ideia is to study such perturbations of interpolation formulae for band-limited and Schwartz functions through functional analysis. Indeed, most of our considerations are based off the idea that, 
whenever an operator $T : B \to B,$  where $B$ is a Banach space, satisfies that 
\[
\|T-I\|_{B \to B} < 1,
\]
then $T$ is, in fact, a \emph{bijection} with continuous inverse $T^{-1} : B \to B.$ In fact, in all our considerations on interpolation formulae below, some form of this principle will be employed, and even the importance of 
other proofs and results in the paper, such as Theorem \ref{eq improvement decay}, arise naturally when trying to employ this principle to different contexts. 

\subsection{Perturbations and Interpolation formulae in the band-limited case} The question of when we are able to recover the values of a function such that its Fourier transform is supported in $[-1/2,1/2]$ from 
its values at $n + \eps_n$ is well-known, having been asked by Paley and Wiener \cite{PW}, where the authors prove that recovery -- and also an associated interpolation formula -- is possible as long as 
$\sup_n |\eps_n| < \pi^{-2}.$ Many results relate to the original problem of Paley and Wiener, but the most celebrated of them all is the so-called Kadec-$1/4$ theorem, 
which states that, as long as $\sup_n |\eps_n| < \frac{1}{4},$ then one can recover any $f \in L^2(\R)$ which has Fourier support on $[-1/2,1/2]$ from its values at $n + \eps_n, \, n \in \Z.$; see \cite{K14} 
for Kadec's original proof and \cite{ALV16} for a generalization.

Our first results provide one with a simpler proof of a particular range of Kadec's result. 

\begin{theorem}\label{th:perturbed_PW}
Let $\{\varepsilon_k\}_{k\in\Z}$ be a sequence of real numbers and consider $L=\sup_k|\varepsilon_k|$. If $L<1/2$ and 
\begin{align*}
    1-\frac{\sin(\pi L)}{\pi L}+\frac{\pi}{3}\frac{L\sin\pi L}{1-L}+\sin\pi L<1,
\end{align*}
then any function $f \in PW_{\pi}$ is completely determined by its values $\{f(n+\eps_n)\}_{n \in \Z},$ and there is $C=C(L)>0$ such that
\begin{align*}
    \frac{1}{C}\sum_{n\in\Z}|f(n+\varepsilon_n)|^2
    \leq\|f\|_2^2\leq    {C}\sum_{n\in\Z}|f(n+\varepsilon_n)|^2,
\end{align*}
for all $f \in PW_{\pi}.$

Moreover, there are functions $g_n\in PW_\pi(\R)$ such that for every $f\in PW_\pi$, the following identity holds:
\begin{align*}
    f(x)=\sum_{n\in\Z}f(n+\varepsilon_n)g_n(x),
\end{align*}
where the right-hand side converges absolutely. 
\end{theorem}

The condition in Theorem \ref{th:perturbed_PW} is satisfied for $L < 0.239,$ which possesses only a $0.011$ gap to Kadec's result. The main difference, however, that while Kadec's proof relies on a clever expansion of the underlying functions in a different 
orthonormal basis, we have almost not used orthogonality in our considerations. We have, nonetheless, chosen not to pursue the path of exploring orthogonality in this question much deeper in order not to make the exposition longer.

We must also remark that, in the proof of such a result, one can use complex numbers for perturbations. The difference is that we have to take into account the sine of complex numbers, and the result would be $L<0.2125$ instead of $L<0.239$. This only falls very mildly short of the results in  \cite[Theorem 3]{ALV16}, where $L<0.218$ is achieved in the complex setting, and our methods 
of proof are relatively simpler in comparison to those of \cite{ALV16}, where the authors must enter the realm of Lamb-Oseen functions and constants. Also, we do not make any use of the orthogonality, which could be exploited to improve on the current result. 

\medskip

As another application of the idea of inverting an operator, we mention a couple of results related to Vaaler's interpolation formula. In \cite{Vaaler}, J. Vaaler proved, as means to study extremal problems in Fourier analysis, the following counterpart to 
the Shannon--Whittaker interpolation formula: let $f \in L^2(\R),$ and suppose that $\widehat{f}$ is supported on $[-1,1].$ Then 
\begin{equation}\label{eq vaaler interpol}
     f(x)=\frac{\sin^2(\pi x)}{\pi^2}\sum_{k\in\Z}\left\{\frac{f(k)}{(x-k)^2}+\frac{f'(k)}{x-k}\right\}.
\end{equation}
This can be seen as a natural tradeoff: \eqref{eq shannon form} demands that we have information at $\frac{1}{2} \Z$ in order to recover the functions $f$ as stated above. On the other hand, Vaaler's result only demands information at $\Z,$ but one must pay the price of also providing it 
for the derivative. 

The first result concerning \eqref{eq vaaler interpol} is a \emph{direct} deduction of its validity from the Shannon--Whittaker formula \eqref{eq shannon form}. We state it, for completeness, in the following form.

\begin{theorem}\label{thm shannon-to-vaaler} Fix a sequence $\{a_k\}_{k \in \Z} \in \ell^2(\Z).$ Consider the function $f \in PW_{\pi}$ given by 
\[
f(x) = \sum_{n \in \Z} a_n \sinc(x-n),
\]
for each $x \in \R.$ Then the interpolation formula 
\[
f(x) = \frac{4\sin^2(\tfrac{1}{2}\pi x)}{\pi^2}\sum_{j\in\Z}\left\{\frac{a_{2k}}{(x-2k)^2}+\frac{b_{2k}}{x-2k}\right\}
\]
holds, where the right-hand side converges uniformly on compact sets, and we let 
\[
b_{k} = \sum_{j\neq k}\frac{a_j}{k-j}(-1)^{k-j}.
\]
\end{theorem}
As a main difference between our proof of Theorem \ref{thm shannon-to-vaaler} and the original proof in \cite{Vaaler} is the absence of any significant use of the Fourier transform. Differently, however, from 
the de Branges spaces approach in \cite{Goncalves}, we do not delve deeply into any theory of function spaces, but rather we make use of classical operators in $\ell^2(\Z)$ such as discrete Hilbert transforms and 
its properties. 

\medskip

Our final contribution in the realm of interpolation formulae for band-limited function is an appropriate perturbation of Vaaler's formula \eqref{eq vaaler interpol}. We mention that, to the best of our knowledge, this 
result in its present form is new. See, for instance, the remark following Corollary 2 in \cite{Goncalves} together with \cite{LyuSeip, SeipOC} for related discussion on sampling sequences with derivatives for $PW_{\pi}.$ 

\begin{theorem}\label{thm vaaler perturb} Let $\{\varepsilon_k\}_{k\in\Z}$ be a sequence of real numbers and consider $L=\sup_k|\varepsilon_k|$. Suppose that $L < 0.111.$ Then any function $f \in PW_{2\pi}$ is completely determined by its values $\{f(n+\eps_n)\}_{n \in \Z}$ and those of its 
derivative $\{f'(n + \eps_n)\}_{n \in \Z},$ and there is $C=C(L)>0$ such that
\begin{align*}
    \frac{1}{C}\sum_{n\in\Z}\left(|f(n+\varepsilon_n)|^2 + |f'(n+\eps_n)|^2\right)
    \leq\|f\|_2^2\leq    {C}\sum_{n\in\Z}|\left(|f(n+\varepsilon_n)|^2 + |f'(n+\eps_n)|^2\right),
\end{align*}
for all $f \in PW_{2\pi}.$ 

Moreover, there are functions $g_n, h_n \in PW_{2\pi}$ so that, for all $f \in PW_{2\pi}$, we have 
\[
f(x) = \sum_{n \in \Z} \left\{f(n+\eps_n) g_n(x) +  f'(n + \eps_n) h_n(x)\right\},
\]
where convergence holds absolutely. 
\end{theorem}

This result and its method of proof follow, essentially, the same basic ideas from Theorem \ref{th:perturbed_PW} and its proof, with only an increase in technical difficulties, such as considering higher order analogues of the 
perturbed discrete Hilbert transforms we use for the proof of \ref{th:perturbed_PW}. We note also that these technical changes, together with the work of Littman \cite{Lit02}, allow one to extend the perturbation results for arbitrarily 
many derivatives; see Theorem \ref{thm derivatives bl} for a discussion on that. In order to avoid the not so pleasant computations needed in order to prove such a result, and due to the fact that its proof follows the main ideas of the proofs of 
theorems \ref{thm vaaler perturb} and \ref{th:perturbed_PW}, we omit it. 

\subsection{Perturbations of symmetric interpolation formulae} Moving on from band-limited functions to Schwartz functions instead, we face the fundamental question of determining whether formula \eqref{eq:interpolation_schwartz} 
is rigid for its interpolation nodes or not. In other words, a fundamental question concerns conditions when we can replace \emph{a single} interpolation node $\sqrt{k}$ by a suitable perturbation of it, say $\sqrt{k+\eps_k},$ where $\eps_k \in (-1,1).$ 

Perhaps surprisingly, the idea of inverting an operator $T$ when it is reasonably close to the identity still works in this context. The next result can thus be regarded as the main new feature of this paper, establishing criteria when we are allowed, not only 
to perturb one node in the interpolation formula, but all of them \emph{simultaneously}. 

\begin{theorem}\label{thm mainthm} There is $\delta>0$ so that, for each sequence of real numbers $\{\eps_k\}_{k \ge 0}$ such that $\eps_k \in (-1/2,1/2), \eps_0 = 0, \, \sup_{k \ge 0} |\eps_k| (1+k)^{5/4} < \delta \, \forall \, k \ge 0,$ 
there are sequences of functions $\{\theta_j\}_{j \ge 0},\{\eta_j\}_{j \ge 0}$ with 
\[
|\theta_j(x)| + |\eta_j(x)| + |\widehat{\theta}_j(x)| + |\widehat{\eta}_j(x)| \lesssim (1+j)^{\mathcal{O}(1)} (1+|x|)^{-10}
\]
and
\[
f(x) = \sum_{j \ge 0} \left( f(\sqrt{j+\eps_j}) \theta_j(x) + \widehat{f}(\sqrt{j+\eps_j}) \eta_j(x)\right),
\]
for all $f \in \mathcal{S}_{even}(\R)$ real-valued functions.
\end{theorem}
In other words, we can perturb each interpolation node from $\sqrt{k}$ to $\sim \sqrt{k+k^{-5/4}}$ and still obtain a valid interpolation formula converging for all Schwartz functions. In fact, one does not striclty need that $f \in \mathcal{S}(\R),$ but
only that $f, \widehat{f}$ decay at least as fast as $(1+|x|)^{-M}$ for some sufficiently large $M \gg 1.$ 

As an immediate corollary of Theorem \ref{thm mainthm}, we obtain that the continuous family of measures
\[
\mu_x = \frac{\delta_x + \delta_{-x}}{2} - \sum_{j \ge 0} \frac{\theta_j(x)}{2} \, \delta_{\pm\sqrt{j+\eps_j}} 
\]
possesses Fourier transform given by 
\[
\widehat{\mu_x} = \sum_{j \ge 0} \frac{\eta_j(x)}{2} \delta_{\pm \sqrt{j+\eps_j}},
\]
whenever $\{\eps_i\}_{i \ge 0}$ satisfies the hypotheses of Theorem \ref{thm mainthm}. This follows from the fact that $\mu_x$ is even and real-valued, so that its distributional Fourier transform will also be an even and real-valued distribution. Therefore, it sufficies to test against even, real-valued functions $f,$ and thus Theorem \ref{thm mainthm} 
gives us the asserted equality. This provides one with a new class of nontrivial examples of \emph{crystalline measures} supported on both space and frquency on basically any set of the form $\pm \sqrt{k + \eps_k}, \, |\eps_k| \le \delta k^{-5/4}.$ This, in particular, aligns well with the recent 
examples from \cite{BRS} and \cite{KurasovSarnak}, which indicate that crystalline measures are, if not impossible, very hard to classify. 

\medskip

In order to prove Theorem \ref{thm mainthm}, we need to find a suitable space to use the idea of inverting operators close to the identity. It turns out that, in analogy to Sobolev spaces, the weighted spaces 
$\ell^2_s(\N)$ of sequences square summable against $n^s$ are natural candidates to work with, as it is well suited to accommodate the sequence 
$$\{(f(\sqrt{k+\eps_k}),\widehat{f}(\sqrt{k+\eps_k}))\}_{k \ge 0}$$
whenever $f, \widehat{f}$ decay sufficiently fast. In order to prove \emph{some} perturbation result -- that is, a weaker version of Theorem \ref{thm mainthm} --, 
using the spaces $\ell^2_s(\N)$ together with the polynomial growth bounds on $\{a_n\}_{n\ge 0}$ from \eqref{eq:interpolation_schwartz} is already enough. 

On the other hand, the fact that me may push the perturbations up until the $k^{-5/4}$ threshold needs a suitable refinement to the Radchenko--Viazovska \cite{RV} or even to the Bondarenko--Radchenko--Seip \cite{BRS} bounds. The next result, 
thus, provides us with an additional \emph{exponential factor} that mitigates growth of the interpolating functions. 

\begin{theorem}\label{eq improvement decay} Let $b_n^{\pm} = a_n \pm \widehat{a_n},$ where $\{a_n\}_{n\ge 0}$ are the basis functions in \eqref{eq:interpolation_schwartz}. Then there is an absolute constant $c>0$ such that 
\begin{align*}
|b_n^{\pm}(x)| & \lesssim n^{1/4} \log^{3/2}(1+n)e^{-c\frac{|x|}{\sqrt{n}}}, \cr 
|(b_n^{\pm})'(x)| & \lesssim n^{3/4} \log^{3/2}(1+n) e^{-c\frac{|x|}{\sqrt{n}}}, \cr 
\end{align*}
for all positive integers $n \in \N.$ 
\end{theorem}

The proof of such a result employs a mixture of the main ideas for the uniform bounds in \cite{RV} and \cite{BRS}, with the addition of an explicit computation of the best uniform constant bounding $|x|^k \left|b_n^{\pm}(x) + (b_n^{\pm})'(x)\right|$ 
in terms of $k$ and $n.$ In order to obtain such a constant, we employ ideas from characterizations of Gelfand--Shilov spaces, as in \cite{CCK}. Finally, with a modification of the growth lemma for 
Fourier coefficients of $2-$periodic functions, we are able to obtain a slight improvement over the growth stated in Theorem \ref{eq improvement decay}. As, however, this modification does not yield any improvement 
on the perturbation range stated in Theorem \ref{thm mainthm}, we postpone a more detailed discussion about it to Corollary \ref{thm bound best} below. 

\medskip 

\subsection{Applications} As a by-product of our method of proof for Theorem \ref{thm mainthm}, we are able to deduce some interesting consequences in regard to some other interpolation formulae and uniqueness results. 

Indeed, it is a not so difficult task to adapt the ideas employed before to the contexts of interpolation formulae for \emph{odd} functions. As remarked by Radchenko and Viazovska, the following interpolation formula is available whenever $f: \R \to \R$ 
is odd and belongs to the Schwartz class: 
\[
f(x) = d_0^+(x) \frac{f'(0) + i \widehat{f}'(0)}{2} + \sum_{n \ge 1} \left( c_n(x) \frac{f(\sqrt{n})}{\sqrt{n}} - \widehat{c_n}(x) \frac{\widehat{f}(\sqrt{n})}{\sqrt{n}} \right),
\]
where the interpolating sequence $\{c_i\}_{i \ge 0}$ possesses analogous properties to those of $\{a_i\}_{i \ge 0},$ and the function $d_0^+(x) = \frac{\sin(\pi x^2)}{\sinh(\pi x)}$ is odd, real and so that 
it vanishes together with its Fourier transform at $\pm \sqrt{n}, \, n \ge 0.$ 

With our techniques, we are able to prove an analogous result to Theorems \ref{eq improvement decay} and \ref{thm mainthm} for the odd interpolation formula. Also, with our techniques, we are able to perturb the Cohn--Kumar--Miller--Radchenko--Viazovska interpolation
results with derivatives in dimensions $8$ and $24$ in a suitable range, as polynomial growth bounds for such interpolating functions are available in \cite{CKMRV2}; see theorems \ref{thm perturb derivatives} and 
\ref{thm perturb odd} for more details. 

Another interesting application of our techniques delves a little deeper into functional analysis techniques. Indeed, in order to prove that the operator that takes the set of values $\{f(\sqrt{k})\}_{k \ge 0},\, \{\widehat{f}(\sqrt{k})\}_{k \ge 0}$ 
to the sequences 
$$\{f(\sqrt{k+\eps_k})\}_{k \ge 0}, \, \{\widehat{f}(\sqrt{k+\eps_k})\}_{k \ge 0}$$
is bounded and close to the identity on a suitable $\ell^2_s(\N) \times \ell^2_s(\N)$ space, we explore two main options, which are \emph{Schur's test} and
the \emph{Hilbert--Schmidt test}. Although there is no direct relation between them, Schur's test seems to hold, in generic terms, for more operators than the Hilbert--Schmidt test, and for that reason we employ the former in our proof of Theorem \ref{thm mainthm}. On the other
hand, the Hilbert--Schmidt test has the advantage that, whenever an operator is bounded in the Hilbert--Schmidt norm, it is automatically a \emph{compact} operator. This allows us to use many more tools derived from the theory 
of Fredholm operators, and, in particular, deduce a sort of interpolation/uniqueness result in case $\eps_0 \neq 0,$ which is excluded by Theorem \ref{thm mainthm} above; see Theorem \ref{thm alternative} below for such an application.  

\medskip 

The perhaps most interesting and nontrivial application of Theorem \ref{thm mainthm} and its techniques is to the problem of \emph{Fourier uniqueness for powers of integers}. In \cite{RS1}, we have proven a preliminary result on conditions on $(\alpha,\beta),$ 
$0 < \alpha,\beta, \, \alpha+\beta<1,$ so that the only $f \in \mathcal{S}(\R)$ such that 
\[
f(\pm n^{\alpha}) = \widehat{f}(\pm n^{\beta}) = 0
\]
is $f \equiv 0.$ In particular, we prove that, if $\alpha = \beta,$ then we can take $\alpha < 1 - \frac{\sqrt{2}}{2}.$ 

By an approximation argument, a careful analysis involving Laplace transforms and the perturbation techniques and results above, we are able to reprove such a result for $\alpha = \beta$ in the $\alpha < \frac{2}{9}$ range in case 
$f$ is real and even by a completely different method than that in \cite{RS1}. Although the current method does not yield any improvement over \cite[Theorem~1]{RS1}, we believe it is a promising path towards proving that the wished uniqueness result 
holds in the $0 < \alpha, \beta < \frac{1}{2}$ range. We refer the reader to Corollary \ref{corol diagonal} below and the discussion that succeeds it for more precise statements. 

\subsection{Organization} We comment briefly on the overall display of our results throughout the text. In Section \ref{sec prelim} below, we discuss generalities on background results needed for the proofs of the main Theorems, going over results in the 
theory of band-limited functions, modular forms and functional analysis. Next, in Section \ref{sec BL}, we prove, in this order, theorems \ref{th:perturbed_PW}, \ref{thm shannon-to-vaaler} and \ref{thm vaaler perturb} about band-limited 
perturbed interpolation formulae. We then prove, in Section \ref{perturbed radchenko viazovska}, Theorem \ref{thm mainthm}, by first discussing the proof Theorem \ref{eq improvement decay} in \S \ref{sec improve}. We then discuss the applications of our main results and techniques 
in Section \ref{sec applications}, and finish the manuscript with Section \ref{sec final}, talking about some possible refinements and open problems that arise from our discussion throughout the paper. 

\section{Preliminaries}\label{sec prelim}

\subsection{Band-limited functions} We start by recalling some basic facts about band-limited functions. Given a function $f \in L^2(\R),$ we say that it is \emph{band-limited} if its Fourier transform satisfies that 
$\supp(\widehat{f}) \subset [-M,M]$ for some $M>0.$ In this case, we say that $f$ is \emph{band-limited to $[-M,M].$}

It is a classical result due to Paley and Wiener that a function $f \in L^2(\R)$ is band-limited if and only if it is the restriction of an entire function $F: \C \to \C$ 
to the real axis, and the function $F$ is of exponential type; that is, there exists $\sigma > 0$ so that, for each $\eps > 0,$ 
\[
|F(z)| \le C_{\eps} e^{(\sigma + \eps)|z|},
\]
for all $z \in \C.$ From now on we will abuse notation and let $F = f$ whenever there is no danger of confusion, and we may also write $f \in PW_{\sigma}$ (Paley--Wiener space) to denote the 
space of functions with such properties. 

Besides this fact, we will make use of some interpolation formulae for those functions. Namely, 
\begin{enumerate}
 \item \textit{Shannon--Whittaker interpolation formula.} For each $f \in L^2(\R)$ band-limited to $[-\frac{1}{2},\frac{1}{2}],$ the following formula holds: 
 \[
 f(x) = \sum_{n \in \Z} f(n) \sinc(x-n),
 \]
 where $\sinc(x) = \frac{\sin(\pi x)}{\pi x}$ and the sum above converges both in $L^2(\R)$ and uniformly on compact sets of $\C.$ 
 \item \textit{Vaaler interpolation formula.} For each $f \in L^2(\R)$ band-limited to $[-1,1],$ the following formula holds: 
 \[
 f(x) = \left(\frac{\sin \pi x}{\pi}\right)^2 \sum_{n \in \Z} \left[ \frac{f(n)}{(x-n)^2} - \frac{f'(n)}{x-n}\right],
 \]
 where the right-hand side converges both in $L^2(\R)$ and uniformly on compact sets of $\C.$ 
\end{enumerate}
For more details on these classical results, see, for instance, \cite{Vaaler}, \cite{Lit02},\cite{PW},\cite{Shannon} and \cite{Whittaker}. 

\subsection{Modular forms} In order to prove the improved estimates on the interpolation basis for the Radchenko--Viazovska interpolation result, we will need to make careful computations 
involving certain modular forms defining the interpolating functions. For that purpose, we gather some of the facts we will need in this subsection. 

We denote by $\HH = \{ z \in \C \colon \text{Im}(z) > 0\}$  the upper half plane in $\C.$ The special feature of this space is that the group $SL_2(\R)$ of matrices with real coefficients and determinant 1 acts naturally on it through M\"obius transformations: for 
\[
\gamma = \begin{pmatrix}
         a & b \\ c & d 
         \end{pmatrix} \in SL_2(\R), \, z \in \HH \Rightarrow \gamma z = \frac{az + b}{cz+d} \in \HH. 
\]
For our purposes, it will suffice to look at the subgroup $PSL_2(\Z) = SL_2(\Z)\slash\{\pm I\}.$ Some elements of this group will be of special interest to us. Namely, we let 
\[
I = \begin{pmatrix} 
    1 & 0 \\ 0 & 1 
    \end{pmatrix}, 
    \,\, T = \begin{pmatrix}
              1 & 1 \\ 0 & 1 
             \end{pmatrix}, \,\, S = \begin{pmatrix}
				  0 & -1 \\ 1 & 0
				  \end{pmatrix}
\]
This already allows us to define the most valuable subgroup of $SL_2(\Z)$ for us: the group $\Gamma_{\theta}$ is defined then as the subgroup of $SL_2(\Z)$ generated by 
$S$ and $T^2.$ This group has $1$ and $\infty$ as cusps, and its standard fundamental domain is given by 
\[
\mathcal{D} = \{z \in \HH \colon |z| > 1, \text{Re}(z) \in (-1,1)\}.
\]
With these at hand, we define \emph{modular forms} for $\Gamma_{\theta}$. For that purpose, we will use the following notation for the Jacobi theta series: 
\[
\vartheta (z,\tau) = \sum_{n \in \Z} \exp(\pi i n^2 \tau + 2\pi i n z).
\]
We are interestes in some of its \emph{Nullwerte}, the so-called Jacobi theta series. These are defined in $\HH$ by
\begin{align*} 
\Theta_2(\tau) &= \exp\left(\frac{\pi i \tau}{4}\right)\vartheta\left(\frac{1}{2}\tau,\tau\right),\cr
\Theta_3(\tau) &= \vartheta(0,\tau) (=:\theta(\tau)), \cr 
\Theta_4(\tau) &= \vartheta\left(\frac{1}{2},\tau\right).\cr
\end{align*}
These functions satisfy the identity $\Theta_3^4 = \Theta_2^4 + \Theta_4^4.$ Moreover, under the action of the elements $S$ and $T$ of $SL_2(\Z),$ they transform as 
\begin{align}\label{eq theta transform}
(-iz)^{-1/2} \Theta_2(-1/z) &= \Theta_4(z), \,\,\Theta_2(z+1) = \exp(i\pi/4) \Theta_2(z), \cr 
(-iz)^{-1/2} \Theta_3(-1/z) &= \Theta_3(z), \,\,\Theta_3(z+1) = \Theta_4(z), \cr 
(-iz)^{-1/2} \Theta_4(-1/z) &= \Theta_2(z), \,\,\Theta_4(z+1) = \Theta_3(z). \cr 
\end{align}
These functions allow us to construct the classical lambda modular invariant given by 
\[
\lambda(z) = \frac{\Theta_2(z)^4}{\Theta_3(z)^4}.
\]
Using the nome $q=q(z) = e^{\pi i z},$ the lambda invariant can be alternatively rewritten as 
\begin{equation}\label{eq lambda}
 \lambda(z) = 16 q \times \prod_{k=1}^{\infty} \left( \frac{1+q^{2k}}{1+q^{2k-1}}\right)^8 = 16q - 128q^2 + 704q^3 + \cdots. 
\end{equation}
Besides this alternative formula, this is also invariant under de action of elements of the subgroup $\Gamma(2) \subset SL_2(\Z)$ of all matrices $\begin{pmatrix} 
                                                                                                                                                   a & b \\ c & d 
                                                                                                                                                   \end{pmatrix}$ so that 
$a\equiv b \equiv 1 \mod 2, \, c \equiv d \equiv 0 \mod 2.$ Besides this invariance, \eqref{eq theta transform} gives us immediately that 
\begin{align}\label{eq lambda transform}
\lambda(z+1) = \frac{\lambda(z)}{1-\lambda(z)}, \, \, \lambda\left(-\frac{1}{z}\right) = 1-\lambda(z).
\end{align}
We then define the modular invariant function for $\Gamma_{\theta}$ to be 
\[
J(z) = \frac{1}{16} \lambda(z) (1-\lambda(z)).
\]
From \eqref{eq lambda transform}, we obtain immediately that $J$ is invariant under the action of elements of $\Gamma_{\theta};$ i.e., 
\[
J(z+2) = J(z), \,\, J\left(-\frac{1}{z}\right) = J(z).
\]
Other properties of the functions $\lambda$ and $J$ that we may eventually need will be proved throughout the text. 

Finally, we mention that, for the proof in \S \ref{perturbed radchenko viazovska}, we will need to use the so-called \emph{$\theta-$automorphy factor} defined, for $z \in \HH$ and $\gamma \in \Gamma_{\theta},$ as 
\[
j_{\theta}(z,\gamma) = \frac{\theta(z)}{\theta(\gamma z)}. 
\]
With this in hands, we defined a slash operator of weight $k/2$ to be 
\[
(f|_{k/2} \gamma)(z) = j_{\theta}(z,\gamma)^k f \left(\frac{az+b}{cz+d}\right),
\]
where $\gamma = \begin{pmatrix} 
                a & b \\ c & d
                \end{pmatrix}.$ These slash operators induce other \emph{sign} slash operators given by 
\[
(f|_{k/2}^{\eps} \gamma) = \chi_{\eps}(\gamma) (f|_{k/2} \gamma),
\]
where we let $\chi_{\eps}$ be the homomorphism of $\Gamma_{\theta}$ so that $\chi_{\eps}(S) = \eps, \chi_{\eps}(T^2)=1.$ 

For more information on the functions $\lambda,J$ and the automorphy factors we just defined, we refer the reader to \cite{Chandrasekharan}and \cite[Section~2]{RV}; see also\cite{BN}, \cite{Zagier}.

\subsection{Functional analysis} We also recall some classical facts in functional analysis that will be useful throughout our proof. 

As our main goal and strategy throughout this manuscript is to prove that a small perturbation of the identity is invertible, we must find ways to prove that the operators arising in our computations are bounded. 
To that extent, we use two major criteria to prove boundedness -- and therefore to prove smallness of the bounding constant. These are:

\begin{enumerate}
 \item \textit{Hilbert-Schmidt test.} Let $H$ be a Hilbert space, and let there be given a linear operator $T:H \to H.$ If $T$ satisfies additionally that 
 \[
 \sum_{i,j} |\langle Te_j,e_i\rangle|^2 < + \infty
 \]
 for some orthonormal basis $\{e_i\}_{i \in \Z}$ of $H,$ then the operator $T$ is bounded. Moreover, 
 \[
 \|T\|_{H \to H}^2 \le \sum_{i,j} |\langle Te_j,e_i\rangle|^2 =: \|T\|_{HS}^2.
 \]
 \item \textit{Schur test.} Let $(a_{ij})_{i,j \ge 0}$ denote an infinite matrix. Suppose that there are two sequences $\{p_i\}_{i \ge 0}$ and $\{q_i\}_{i \ge 0}$ of 
 positive real numbers so that 
 \begin{align*}
  \sum_{i \ge 0} |a_{ij}|q_i &\le \lambda p_j, \cr 
  \sum_{j \ge 0} |a_{ij}|p_j &\le \mu q_i,\cr 
 \end{align*}
 for some positive constants $\mu,\lambda > 0.$ Then the operator $T:\ell^2(\N) \to \ell^2(\N)$ given by $a_{ij} = \langle Te_i,e_j \rangle$ (where $\{e_i\}_{i \ge 0}$ denotes the 
 standard orthonormal basis of $\ell^2(\N)$) extends to a \emph{bounded} linear operator. Moreover, 
 \[
 \|T\|_{\ell^2 \to \ell^2} \le \sqrt{\mu \lambda}. 
 \]
\end{enumerate}

Both tests will play a major role in the deduction of the validity of perturbed interpolation versions of the Radchenko--Viazovska result. The main difference is that, while Schur's test 
generally gives one boundedness for more operator, the Hilbert-Schmidt test imposes stronger conditions on the operator. In fact, let us denote by $T \in \mathcal{HS}(H)$ the fact that $\|T\|_{HS} < +\infty.$ A classical consequence of this fact is that $T \in \mathcal{K}(H);$ that is, $T$ is compact. 

This fact will be used when proving that a suitable version of our interpolation results holds for small perturbations of the origin. See, for instance, \cite[Chapter~6]{Brezis}

\subsection{Notation} We will use Vinogradov's modified notation throughout the text; that is, we write $A \lesssim B$ in case there is an absolute constant $C>0$ so that $A \le C \cdot B.$ If the constant $C$ before depends on some set of parameters 
$\lambda,$ we shall write $A \lesssim_{\lambda} B.$ 

On the other hand, we shall also use the big-$\mathcal{O}$ notation $f = \mathcal{O}(g)$ if there is an absolute constant $C$ such that $|f| \le C \cdot g,$ although the usage of this will be restricted mostly to sequences. We may occasionally 
use as well the standard Vinogradov notation $a \ll b$ to denote that there is a (relatively) \emph{large} constant $C > 1$ such that $a \le C \cdot b.$ 

We shall also denote the spaces of sequences decaying polinomially as 
\[
\ell^2_s(\N) = \left\{ \{a_n\}_{n \in \N} \colon \sum_{n \in \N} |a_n|^2 n^{2s} < +\infty \right\}.
\]
Finally, we always normalize our Fourier transform as 
$$\widehat{f}(\xi) = \mathcal{F}f(\xi) = \int_{\R^n} f(x) \, e^{-2 \pi i x \cdot \xi} \, \mmd x.$$ 

\section{Perturbed Interpolation Formulae for Band-Limited functions}\label{sec BL}

\subsection{Perturbed forms of the Shannon--Whittaker formula and Kadec's result} 

Fix a sequence $\mathbf{\varepsilon}=\{\varepsilon_k\}_{k\in\Z}$ of real numbers such that $\sup_k|\varepsilon_k|<1$. We wish to obtain a criterion based solely on the value of $L=\sup_n|\varepsilon_n|$ such that the sequence $\{n+\varepsilon_n\}_{n\in\Z}$ is completely interpolating in $\mathrm{PW}_{\pi}$, i.e, for every sequence $a=\{a_n\}\in\ell^2(\Z)$ there is a unique $f\in L^2(\R)$  of exponential type $\tau(f)\leq\pi$ that satisfies
\begin{align*}
    f(n+\varepsilon_n)=a_n.
\end{align*}
Our goal here is to obtain a simple proof of such a criterion going through new and simple ideas. We will fall short of the $1/4$ proven by Kadec by approximately 0.11, but it illustrates the power of our perturbation scheme and does not go through the theory of exponential bases. 

In this particular case, we need to invert in $\ell^2(\Z)$ the operator given by
$$A_\varepsilon(a)(n)=\sum_{k\in\Z}a_k\sinc(n+\varepsilon_n-k),$$
where
\begin{align*}
    \sinc(x)=\frac{\sin\pi (x)}{\pi x}.
\end{align*}
The fact $A_\varepsilon$ is invertible will follow from proving that it is a close perturbation of the identity whenever $L$ is sufficiently small.

\subsubsection{Auxiliary perturbations of the Hilbert transforms}

Given a sequence $a=\{a_k\}_{k\in\Z}$, we define the following operators, which are kin to the discrete Hilbert transform:
\begin{align*}
    \H_\varepsilon(a)(n)&=\sum_{k\neq n}\frac{(-1)^{n-k}a_{k}}{n+\varepsilon_n-k},  \\
    \H_0(a)(n)&=\sum_{k\neq n}\frac{(-1)^{n-k}a_{k}}{n-k}.
\end{align*}
We start by comparing these two objects:
\begin{align*}
    \H_0(a)(n)-\H_\varepsilon(a)(n)&=\sum_{k\neq n}(-1)^{n-k}a_{k}\left(\frac{1}{n-k}-\frac{1}{n+\varepsilon_n-k}  \right)\\
    &=\varepsilon_n\sum_{k\neq n}(-1)^{n-k}a_{k}\frac{1}{(n-k)(n+\varepsilon_n-k)}.
\end{align*}
This identity then gives us
\begin{align*}
    |\H_0(a)(n)-\H_\varepsilon(a)(n)|&\leq|\varepsilon_n|\sum_{k\neq n}|a_{k}|\frac{1}{|n-k|^2}\frac{|n-k|}{|n+\varepsilon_n-k|}\\
    &\leq \frac{|\varepsilon_n|}{1-|\varepsilon_n|}\sum_{k\neq n}|a_{k}|\frac{1}{|n-k|^2}.
\end{align*}
This means that, in norm, one can compare these two operators. Indeed, it is a classical result that the operator norm of $\H_0$ is $\pi$, and by Plancherel the operator norm of the transformation
\begin{align*}
    \mathcal{S}(a)=\sum_{k\neq n}a_{k}\frac{1}{|n-k|^2}
\end{align*}
is $\pi^2/3$. This in turn implies
\begin{align}\label{eq:first_norm_estimate}
    \|\H_\varepsilon\|\leq \pi+\frac{\pi^2}{3}\frac{\sup_n|\varepsilon_n|}{1-\sup_n|\varepsilon_n|}. 
\end{align}

\subsubsection{Norm estimates of the perturbation} It is worth noticing the the estimate \eqref{eq:first_norm_estimate} is very crude, as it is meant to depend only on $L=\sup_n|\varepsilon_n|$. For instance, if $\{\varepsilon_n\}_{n \in \Z}$ is a constant sequence, then the norm $\|\H_\varepsilon\|$ is equal to $\pi$. We also note that 
the fact that we obtain invertibility by means of perturbations of small norm of a invertible operator does not take into account other factors, such as cancellation.

In order to apply our perturbation scheme to the operator $A_\varepsilon$, we need to bound the following family of operators:
\begin{align*}
    P_\varepsilon(a)(n)=\sum_{k\in\Z }a_k(\mathrm{sinc}(n+\varepsilon_n-k)-\delta_{n,k}).
\end{align*}

We may rewrite them as 
\begin{align*}\begin{split}
    P_\varepsilon(a)(n)=&(\sinc(\varepsilon_n)-1)a_n+\sum_{k\neq n}a_k(\mathrm{sinc_n}(n+\varepsilon_n-k))\\
    =&(\sinc(\varepsilon_n)-1)a_n +\sum_{k\neq n}a_k\frac{(-1)^{n-k}\sin\pi\varepsilon_n}{\pi(n+\varepsilon_n-k)} \\    \end{split}
\end{align*}
This implies, on the other hand,
\begin{align*}
    P_\varepsilon(a)(n)=(\sinc(\varepsilon_n)-1)a_n +\left(\frac{\sin{\pi\varepsilon_n}}{\pi}\right)\H_\varepsilon(a)(n),
\end{align*}
which in turn implies that 
\begin{align*}
    \|P_\varepsilon\|&\leq \sup_n|\sinc(\varepsilon_n)-1|+\sup_n\left|\frac{\sin{\pi\varepsilon_n}}{\pi}\right|\|\mathcal{H}_\varepsilon\|\\
    &\leq \sup_n|\sinc(\varepsilon_n)-1|+\sup_n|\sin\pi\varepsilon_n|+\frac{\pi}{3}\frac{\sup_n|\sin\pi\varepsilon_n|\sup_n|\varepsilon_n|}{1-\sup_n|\varepsilon_n|}.
\end{align*}
Since $A_\varepsilon=P_\varepsilon+Id$, whenever 
\begin{align*}
    1-\sinc(L)+|\sin\pi L|+\frac{\pi}{3}\frac{L\sin\pi L}{1-L}<1,
\end{align*}
we will have that $A_\varepsilon$ is invertible. In particular, a routine numerical evaluation implies that $L<0.239$ satisfies the inequality above. Let then $A_{\eps}^{-1} : \ell^2(\Z) \to \ell^2(\Z)$ be the inverse of $A_{\eps},$ which is 
continuous by the considerations above. We know, by the Shannon--Whittaker interpolation formula \ref{eq shannon form}, that $A_{\eps}$ takes $\{f(k)\}_{k \in \Z}$, for $f \in PW_{\pi}$, to 
$\{f(k+\eps_k)\}_{k \in \Z}.$ This is enough to prove the assertion about recovery, and as such implies that 
\[
\sum_{n \in \Z} |f(n+\eps_n)|^2 
\]
is an equivalent norm to the usual $L^2-$norm on $PW_{\pi},$ by \cite[Theorem~1.13]{Young}. 

Moreover, by writing 
\[
A_{\eps}^{-1}(b)(k) = \sum_{n \in \Z} b_n \cdot \rho_{k,n}, 
\]
we have immediately 
\begin{equation}\label{eq shannon invert}
\sum_{n \in \Z} f(n+ \eps_n) \rho_{k,n} = f(k), 
\end{equation}
and $\sup_{n} \left( \sum_{k \in \Z} |\rho_{k,n}|^2\right) \lesssim 1.$ If $(A_{\eps}^{-1})^* : \ell^2(\Z) \to \ell^2(\Z)$ denotes the adjoint of the inverse of $A_{\eps},$ then we see that 
\[
\| (A_{\eps}^{-1})^* (\sinc_x(k)) \|_{\ell^2(\Z)} \lesssim \|A_{\eps}^{-1}\|_{\ell^2 \to \ell^2},
\]
where the implicit constant does \emph{not} depend on $x,$ and we let $\sinc_x(k) := \sinc(x-k).$ Therefore, by letting $g_n(x) =  \sum_{k \in \Z} \rho_{k,n}\sinc(x-k)$, 
we have 
\[
\sup_{x \in \R} \left( \sum_{n \in \Z} |g_n(x)|^2 \right)^{1/2} \lesssim 1,
\]
and thus, by the previous considerations,  the sum 
$\sum_{n \in \Z} f(n + \eps_n) g_n(x) $
converges absolutely. As $\langle (A_{\eps}^{-1})^* (\sinc_x(k)), f(n+\eps_n) \rangle =\langle \sinc_x(k), A_{\eps}^{-1}\left(f(n+\eps_n) \right)\rangle = f(x)$ by Shannon--Whittaker, 
this implies 
\[
f(x) = \sum_{n \in \Z} f(n+\eps_n) g_n(x),
\]
as desired. This finishes the proof of Theorem \ref{th:perturbed_PW}. 

\subsection{From Shannon to Vaaler: the proof of Theorem \ref{thm shannon-to-vaaler}} We now concentrate in proving that the usual Shannon--Whittaker interpolation 
formula implies Vaaler's celebrated interpolation result with derivatives \cite{Vaaler}. 

Indeed, as proving that the interpolation formula of Theorem \ref{thm shannon-to-vaaler} converges uniformly on compact sets of $\C$ is a routine computation, given that
$\{a_k\}_{k \in \Z}, \{b_k\}_{k \in \Z} \in \ell^2(\Z),$ we shall omit this part and focus on proving that the asserted equality holds. 

Given a sequence $a=\{a_k\}_{k\in\Z}$, we define the following operators:
\begin{align*}
    \H(a)(k)&=\frac{1}{\pi}\sum_{0\neq j\in\Z}\frac{a_{k-j}}{j} = \frac{1}{\pi}\sum_{k\neq j\in\Z}\frac{a_{j}}{k-j}, \\
    \H_1(a)(k)&=\frac{1}{\pi}\sum_{j\in\Z}\frac{a_{k-j}}{j+\tfrac{1}{2}} =\frac{1}{\pi}\sum_{j\in\Z}\frac{a_{j}}{k-j+\tfrac{1}{2}}.
\end{align*}
It is known that both $\H$ and $\H_1$ are bounded operators in $\ell^2(\Z)$, with $\H_1$ being also unitary with $\H_2$ its inverse being given by
\begin{align*}
    \H_2(a)(k)&=-\frac{1}{\pi}\sum_{j\in\Z}\frac{a_{j-k}}{j-\tfrac{1}{2}} =\frac{1}{\pi}\sum_{j\in\Z}\frac{a_{j}}{j-k+\tfrac{1}{2}}.
\end{align*}

Given a function $f \in PW_{\pi}$, as a consequence of the Shannon--Whittaker interpolation formula we obtain, for every $k\in \Z$, that
\begin{align*}
    f'(k)=\sum_{j\neq k}\frac{f(j)}{k-j}(-1)^{k-j}.
\end{align*}
We consider three sequences, as follows:
\begin{align*}
    a(k)&=f(2k-1), \,\, b(k)=f(2k), \,\,  c(k)=f'(2k). 
\end{align*}
We have, thus,
\begin{align*}
    c(k)=f'(2k)&=\sum_{j\neq 2k}\frac{f(j)}{2k-j}(-1)^{2k-j} =\frac{1}{2}\sum_{j\neq k}\frac{f(2j)}{k-j}-\frac{1}{2}\sum_{j\in\Z}\frac{f(2j-1)}{k-j+\tfrac{1}{2}}     \\
               &=\frac{1}{2}\sum_{j\neq k}\frac{b(j)}{k-j}-\frac{1}{2}\sum_{j\in\Z}\frac{a(j)}{k-j+\tfrac{1}{2}} =\frac{\pi}{2}\H(b)(k)-\frac{\pi}{2}\H_1(a)(k).
\end{align*}
This means that, for every $k \in \Z,$
\begin{align*}
    \H_1(a)(k)=\H(b)(k)-\frac{2}{\pi} c(k).
\end{align*}
Since $\H_2$ is the inverse of $\H_1$, this can be rewritten as 
$$a(k)=(\H_2\circ\H)(b)(k)-\frac{2}{\pi}\H_2(c)(k).$$
We know, by the Shannon--Whittaker interpolation formula, that
\begin{align*}
    f(x)&=\sum_{k\in\Z}f(k)\frac{\sin\pi (x-k)}{\pi (x-k)}.
\end{align*}
This implies, on the other hand,
\begin{align*}
    f(x)=&\sum_{k\in\Z}f(2k)\frac{\sin\pi (x-2k)}{\pi (x-2k)}+ \sum_{k\in\Z}[(\H_2\circ\H)(b)(k)-\frac{2}{\pi}\H_2(c)(k)]\frac{\sin\pi (x-2k+1)}{\pi (x-2k+1)} \\
        =&\sum_{k\in\Z}b(k)\frac{\sin\pi x}{\pi (x-2k)}+ \sum_{k\in\Z}(\H_2\circ\H)(b)(k)\frac{\sin\pi (x-2k+1)}{\pi (x-2k+1)} \\
        &-\frac{2}{\pi}\sum_{k\in\Z}\H_2(c)(k)\frac{\sin\pi (x-2k+1)}{\pi (x-2k+1)} = A(x)+B(x)+C(x).\\
\end{align*}
We shall investigate each term $A,B$ and $C$ thoroughly in order to obtain our final result. 

\subsubsection{Determining $C$}
By considering the family of functions $h_j\in PW_\pi$  -- which satisfy the important property $h_j(k)=0$, if $k\in2\Z$ -- given by
$$h_j(z)=\frac{\sin^2(\tfrac{1}{2}\pi z)}{\pi^2(z-2j)},$$
we obtain 
\begin{align*}
    C(x)&=-2\sum_{k\in\Z}\sum_{j\in\Z}\frac{f'(2j)}{\pi^2(j-k+\tfrac{1}{2})}\frac{\sin\pi (x-2k+1)}{\pi (x-2k+1)} \\
      &=4\sum_{j\in\Z}f'(2j)\sum_{k\in\Z}\frac{1}{\pi^2((2k-1)-2j)}\frac{\sin\pi (x-(2k-1))}{\pi (x-(2k-1))} \\
      &=4\sum_{j\in\Z}f'(2j)\sum_{k\in\Z}h_j(2k-1)\frac{\sin\pi (x-(2k-1))}{\pi (x-(2k-1))} \\
      &=4\sum_{j\in\Z}f'(2j)\sum_{k\in\Z}h_j(k)\frac{\sin\pi (x-k)}{\pi (x-k)}.
\end{align*}
Notice that one can use Fubini's theorem to justify all the changes of order of summation, by the fact that $h_j \in PW_{\pi}$. By applying the Shannon-Whittaker interpolation to $h_j$, we have
$$C(x)=4\sum_{j\in\Z}f'(2j)\frac{\sin^2(\tfrac{1}{2}\pi x)}{\pi^2(x-2j)}$$
\subsubsection{Determining $B$}
For the second term, we expand
\begin{align*}
    B(x)=&\sum_{k\in\Z}\H_2\circ\H(b)(k)\frac{\sin\pi (x-2k+1)}{\pi (x-2k+1)} \\
        =&\frac{1}{\pi}\sum_{k\in\Z}\frac{\sin\pi (x-2k+1)}{\pi (x-2k+1)}\sum_j\frac{\H (b)(j)}{j-k+\tfrac{1}{2}} \\
        =&\frac{1}{\pi^2}\sum_{k\in\Z}\frac{\sin\pi (x-2k+1)}{\pi (x-2k+1)}\sum_j\sum_{l\neq j}\frac{b(l)}{(j-k+\tfrac{1}{2})(j-l)}.
\end{align*}
By Fubini's theorem, this implies
\begin{align*}
    B(x)=&\frac{1}{\pi^2}\sum_{l\in\Z}b(l)\sum_{j\neq l}\frac{1}{j-l}\sum_{k\in\Z}\frac{1}{j-k+\tfrac{1}{2}}\frac{\sin\pi (x-2k+1)}{\pi (x-2k+1)} \\
        =&\frac{1}{\pi^2}\sum_{l\in\Z}b(l)\sum_{j\neq l}\frac{2}{j-l}\sum_{k\in\Z}\frac{1}{2j-2k+1}\frac{\sin\pi (x-2k+1)}{\pi (x-2k+1)} \\
        =&\frac{1}{\pi^2}\sum_{l\in\Z}b(l)\sum_{j\neq l}\frac{2}{j-l}\frac{\sin^2(\tfrac{1}{2}\pi x)}{2j-x} = \frac{\sin^2(\tfrac{1}{2}\pi x)}{\pi^2}\sum_{l\in\Z}b(l)\sum_{j\neq 0}\frac{1}{j(j+l-\tfrac{x}{2})}.
\end{align*}
But it is a well-known fact that the summation formula 
\begin{align*}
    g(z)&=\sum_{j\neq 0}\frac{1}{j(j+z)} =\frac{\psi(1+z)-\psi(1-z)}{z},
\end{align*}
holds, where $\psi(z)=\frac{d}{dz}\log{\Gamma}(z)$ is the digamma function. This implies, on the other hand,
\begin{align*}
    B(x)&=\frac{2\sin^2(\tfrac{1}{2}\pi x)}{\pi^2}\sum_{l\in\Z}b(l)\frac{\psi(1+l-\tfrac{x}{2})-\psi(1-l+\tfrac{x}{2})}{2l-x}.
\end{align*}
\subsubsection{Determining $A+B$}

Using that $\sin(2x)=2\sin x\cos x$, we obtain
\begin{align*}
    A(x)=-\frac{2\sin^2(\tfrac{1}{2}\pi x)}{\pi^2}\sum_{l\in\Z}b(l)\frac{\pi\cot(\pi\tfrac{x}{2} )}{2l-x}.
\end{align*}
The digamma function satisfies the following functional equations, which we shall make use of:
\begin{align*}
    \psi(1-z)&=\psi(z)+\pi\cot\pi z, \\
    \psi(1+z)&=\psi(z)+\frac{1}{z}. 
\end{align*}
Using these relations with $z=\frac{x}{2}-l$ in the equations above, we obtain readily
\begin{align*}
    A(x)+B(x)= \frac{4\sin^2(\tfrac{1}{2}\pi x)}{\pi^2}\sum_{l\in\Z}b(l)\frac{1}{(x-2l)^2}.
\end{align*}
\subsubsection{A+B+C} Summing the analysis undertaken for the terms above, we have
\begin{align*}
    f(x)=A(x)+B(x)+C(x)=\frac{4\sin^2(\tfrac{1}{2}\pi x)}{\pi^2}\sum_{j\in\Z}\left\{\frac{f(2k)}{(x-2k)^2}+\frac{f'(2k)}{x-2k}\right\}. 
\end{align*}
This finishes the proof of Theorem \ref{thm shannon-to-vaaler}.

\subsection{Perturbations of Interpolation Formulae with derivatives} By the arguments in the previous section, the formula we just derived for $PW_{2\pi}$, i.e.,
\begin{align*}
    f(x)=\frac{\sin^2(\pi x)}{\pi^2}\sum_{k\in\Z}\left\{\frac{f(k)}{(x-k)^2}+\frac{f'(k)}{x-k}\right\},
\end{align*}
converges in compact sets of $\C$. We fix, for shortness, the notation
\begin{align*}
    g(x)=\frac{\sin^2(\pi x)}{\pi^2x^2}, h(x)=\frac{\sin^2(\pi x)}{\pi^2x},
\end{align*}
which means we can read Vaaler's interpolation as 
\begin{align*}
     f(x)=\sum_{k\in\Z}\left\{f(k)g(x-k)+f'(k)h(x-k)\right\}.
\end{align*}
Because of uniform convergence, we can differentiate term by term in the above formula. This implies, thus,
\begin{align*}
     f'(x)=\sum_{k\in\Z}\left\{f(k)g'(x-k)+f'(k)h'(x-k)\right\}.
\end{align*}
We record, for completeness, the formulae for the derivatives of $g',\, h':$
\begin{align*}
    g'(x)&=\frac{2\sin(\pi x)(\pi x \cos(\pi x)-\sin(\pi x))}{\pi^2x^3}, \\
    h'(x)&=\frac{\sin(\pi x)(2\pi x \cos(\pi x)-\sin(\pi x))}{\pi^2x^2},
\end{align*}
and for $n\in\Z,$
\begin{align*}
    g_n=h'_n&=0,\,\, g'_n=h_n=\delta_{0}.
\end{align*}
Our goal now is to invert the operator $\mathcal{A}=\mathcal{A}_\varepsilon$ defined in $\ell^2(\Z)\times\ell^2(\Z)$ by 
\begin{align}\label{eq:system_2}
   \mathcal{A}_1(\mathbf{a},\mathbf{b})_n&=\sum_{k\in\Z}a_k\cdot g(n+\varepsilon_n-k)+\sum_{k\in\Z}b_k\cdot h(n+\varepsilon_n-k)\cr
\mathcal{A}_2(\mathbf{a},\mathbf{b})_n&=\sum_{k\in \Z}a_k\cdot g'(n+\varepsilon_n-k)+\sum_{k\in\Z}b_k\cdot h'(n+\varepsilon_n-k),
\end{align} 
where $\mathcal{A}(\mathbf{a},\mathbf{b})=(\mathcal{A}_1(\mathbf{a},\mathbf{b}),\mathcal{A}_2(\mathbf{a},\mathbf{b}))$ for $(\mathbf{a},\mathbf{b})\in \ell^2(\Z)\times\ell^2(\Z)$. 
Furthermore, we wish to establish a criterion that depends only on $L=\sup|\varepsilon_n|$. For that purpose, we estimate when the operator norm of $\mathcal{A}_{\eps}-Id$ 
from $\ell^2(\Z) \times \ell^2(\Z)$ to itself is small, in terms of $L$.

\subsubsection{Auxiliary perturbations for the derivative case}

Given a sequence $\mathbf{a}=\{a_k\}_{k\in\Z}$, we define the following operators:
\begin{align*}
    \H^p_\varepsilon(\mathbf{a})_n&=\sum_{k\neq n}\frac{a_{k}}{(n+\varepsilon_n-k)^p},
\end{align*}
and  denote by $\H^p_0$ the operator associated to the sequence $\varepsilon_n=0, \forall n \in \Z$. In an analogous manner to the proof of Theore m\ref{th:perturbed_PW}, we compare: 
\begin{align*}
    \H^p_0(\mathbf{a})_n-\H^p_\varepsilon(\mathbf{a})_n&=\sum_{k\neq n}a_k\left(\frac{1}{(n-k)^p}-\frac{1}{(n+\varepsilon_n-k)^p}\right)  \cr
                     &=\sum_{j=0}^{p-1}\binom{p}{j}\varepsilon_n^{p-j}\sum_{k\neq n}\frac{a_{k}}{(n+\varepsilon_n-k)^p(n-k)^{p-j}}.         
\end{align*}
Therefore,
\begin{align*}
    |\H^p_0(\mathbf{a})_n-\H^p_\varepsilon(\mathbf{a})_n|
                     &\leq\sum_{j=0}^{p-1}\binom{p}{j}|\varepsilon_n|^{p-j}\sum_{k\neq n}\frac{a_{k}}{|n-k|^{2p-j}}\frac{|n-k|^p}{(|n-k|-|\varepsilon_n|)^p}.         \cr
                     &\leq \frac{1}{(1-|\varepsilon_n|)^p}\sum_{j=0}^{p-1}\binom{p}{j}|\varepsilon_n|^{p-j}\mathcal{S}^{2p-j}(\mathbf{a^*})_n,
\end{align*}
where
\begin{align*}
    \mathcal{S}^q_\varepsilon(\mathbf{a})_n&=\sum_{k\neq n}\frac{a_{k}}{|n-k|^q},
\end{align*}
and $\mathbf{a}^*=(|a_n|)$. Let us consider $S(p)=\max\{\|\mathcal{S}^q\|,\,q=1,\dots,2p\}$. Since $\mathcal{S}^{q+1}(\mathbf{a^*})_n\leq\mathcal{S}^{q}(\mathbf{a^*})_n$, we have
\begin{align*}
  |\H^p_0(\mathbf{a})_n-\H^p_\varepsilon(\mathbf{a})_n|
                     &\leq \frac{\mathcal{S}^{p+1}(\mathbf{a^*})_n}{(1-|\varepsilon_n|)^p}\sum_{j=0}^{p-1}\binom{p}{j}|\varepsilon_n|^{p-j} \cr
                     &=\left(\frac{(1+|\varepsilon_n|)^p-1}{(1-|\varepsilon_n|)^p}\right)\mathcal{S}^{p+1}(\mathbf{a^*})_n.
\end{align*}
This means that we have the following estimate on the norm of the perturbed operator:
\begin{align}\label{eq:H_p_norm_control}
\|\H^p_\varepsilon\|\leq \gamma_p(L),
\end{align}
where we let 
$$\gamma_p(L)=\|\H^p_0\|+\frac{(1+L)^p-1}{(1-L)^p}\|\mathcal{S}^{p+1}\|.$$
Now, in order to estimate the value of $\gamma_p(L),$ we resort to \cite[Corollary 2]{Lit02}, which gives us that
$$\|\H^p_0\|=\frac{(2\pi)^mb_m}{m!},$$ 
where $b_m$ is the maximum of $|B_m(x)|$ when $x\in[0,1]$, and $B_m$ denotes the $m$-th Bernoulli polynomial. Therefore, 
\begin{align*}
\|\H^1_0\|&=\pi, \|\H^2_0\|= \frac{\pi^2}{3}, \|\H^3_0\|= \frac{\pi^3}{9\sqrt{3}}.
\end{align*}
On the other hand, by Plancherel`s theorem it is easy to see that
$$\|\mathcal{S}^p\|= 2\zeta(p).$$
Joining all these data into \eqref{eq:H_p_norm_control}, we obtain
\begin{align}\label{eq:H_p_norm_control2}
\|\H^1_\varepsilon\|&\leq\pi+\left(\frac{L}{1-L}\right)\frac{\pi^2}{3},\cr
\|\H^2_\varepsilon\|&\leq\frac{\pi^2}{3}+2\left(\frac{L^2+2L}{(1-L)^2}\right)\zeta(3) ,\cr
\|\H^3_\varepsilon\|&\leq\frac{\pi^3}{9\sqrt{3}}+\left(\frac{L^3+3L^2+3L}{(1-L)^3}\right)\frac{\pi^4}{45}.
\end{align}

\subsubsection{Norm estimates of the perturbations in the derivative case}

In order to invert the operator $\mathcal{A}_\varepsilon$, we estimate the norm of $\mathcal{P}_\varepsilon=\mathcal{A}_\varepsilon-Id=(\mathcal{P}_1,\mathcal{P}_1)$, where
\begin{align}\label{eq:system_10}
   \mathcal{P}_1(\mathbf{a},\mathbf{b})_n&=\sum_{k\in\Z}a_k\cdot (g(n+\varepsilon_n-k)-\delta_{n,k})+\sum_{k\in\Z}b_k\cdot h(n+\varepsilon_n-k),\cr
\mathcal{P}_2(\mathbf{a},\mathbf{b})_n&=\sum_{k\in \Z}a_k\cdot g'(n+\varepsilon_n-k)+\sum_{k\in\Z}b_k\cdot (h'(n+\varepsilon_n-k)-\delta_{n,k}).
\end{align} 
By a straightforward calculation,
\begin{align}\label{eq:system_11}
   \mathcal{P}_1(\mathbf{a},\mathbf{b})_n=\,&(g(\varepsilon_n)-1)a_n+\frac{\sin(\pi\varepsilon_n)^2}{\pi^2}\H^2_\varepsilon(\mathbf{a})_n + h(\varepsilon_n)b_n+\frac{\sin(\pi\varepsilon_n)^2}{\pi^2}\H^1_\varepsilon(\mathbf{b})_n,\cr
\mathcal{P}_2(\mathbf{a},\mathbf{b})_n=\,&g'(\varepsilon_n)a_n+\frac{2\sin(\pi \varepsilon_n)(\pi \varepsilon_n\cos(\pi\varepsilon_n )-\sin(\pi \varepsilon_n))}{\pi^2}\H^3_\varepsilon(\mathbf{a}) \cr
		&+(h'(\varepsilon_n)-1)b_n+\frac{\sin(\pi \varepsilon_n)(2\pi \varepsilon_n \cos(\pi \varepsilon_n)-\sin(\pi \varepsilon_n))}{\pi^2}\H^2_\varepsilon(\mathbf{b}).
\end{align} 
Thus,  
$$\|\mathcal{P}_\varepsilon\|\leq\sqrt{2}\max\{|g(L)-1|,|h'(L)-1|,|g'(L)|,|h(L)|\}+\frac{\sin(\pi L)^2}{\pi^2}\|\mathcal{G}_\varepsilon\|,$$
where $\mathcal{G}_\varepsilon=\mathcal{G} =(\mathcal{G}_1,\mathcal{G}_2)$ and
\begin{align}\label{eq:system_3}
      \mathcal{G}_1(\mathbf{a},\mathbf{b})_n=& \H^2_\varepsilon(\mathbf{a})_n+\H^1_\varepsilon(\mathbf{b})_n, \cr
   \mathcal{G}_2(\mathbf{a},\mathbf{b})_n=& \frac{2(\pi \varepsilon_n\cos(\pi\varepsilon_n )-\sin(\pi \varepsilon_n))}{\sin(\pi\varepsilon)}\H^3_\varepsilon(\mathbf{a}) +\frac{(2\pi \varepsilon_n \cos(\pi \varepsilon_n)-\sin(\pi \varepsilon_n))}{\sin(\pi\varepsilon)}\H^2_\varepsilon(\mathbf{b}).
\end{align} 
By taking $L<1/4$ and using the Cauchy-Schwarz inequality, we have
\begin{align*}
&\|\mathcal{G}_\varepsilon\|^2/2 \leq \max\{\|\H^1_\varepsilon\|,\|\H^2_\varepsilon\|\}^2 \cr
&+\max\left\{\left(\frac{2(\pi L\cos(\pi L )-\sin(\pi L))}{\sin(\pi L)}\right)^2\|\H^3_\varepsilon\|^2,\left(\frac{(2\pi L \cos(\pi L)-\sin(\pi L))}{\sin(\pi L)}\right)^2\|\H^2_\varepsilon\|^2\right\}  \cr
&\leq\max\{\gamma_1(L)^2,\gamma_2(L)^2\} \cr
&+\max\left\{\left(\frac{2(\pi L\cos(\pi L )-\sin(\pi L))}{\sin(\pi L)}\right)^2\gamma_3(L)^2,\left(\frac{(2\pi L \cos(\pi L)-\sin(\pi L))}{\sin(\pi L)}\right)^2\gamma_2(L)^2\right\}.
\end{align*}
We note that we have abused the notation $\|\mathcal{G}_{\eps}\|$ to denote the operator norm of $\mathcal{G}_{\eps}$ when defined on $\ell^2(\Z) \times \ell^2(\Z).$ One can further check that, for $0\leq L<1/4,$ 
\begin{align*}
|g(L)-1|<|h'(L)-1|, |h(L)|&<|g'(L)|, \gamma_1(L)^2<\gamma_2(L)^2 \text{ and } \cr
\left(\frac{2(\pi L\cos(\pi L )-\sin(\pi L))}{\sin(\pi L)}\right)^2\gamma_3(L)^2 &<\left(\frac{(2\pi L \cos(\pi L)-\sin(\pi L))}{\sin(\pi L)}\right)^2\gamma_2(L)^2 ,
\end{align*}
which means, in turn, 
$$\|\mathcal{G}_\varepsilon\|\leq \gamma_2(L)\sqrt{2\left(1+\left(\frac{(2\pi L \cos(\pi L)-\sin(\pi L))}{\sin(\pi L)}\right)^2\right)},$$
and directly implies the estimate
\begin{align*}
\|\mathcal{P}_\varepsilon\|\leq& 1-\frac{\sin(\pi L)(2\pi L \cos(\pi L)-\sin(\pi L))}{\pi^2L^2}+\frac{2\sin(\pi L)(\sin(\pi L)-\pi L \cos(\pi L))}{\pi^2L^3}\cr
+\frac{\sin(\pi L)^2}{\pi^2} &\left(\frac{\pi^2}{3}+2\left(\frac{L^2+2L}{(1-L)^2}\right)\zeta(3)\right)\sqrt{2\left(1+\left(\frac{(2\pi L \cos(\pi L)-\sin(\pi L))}{\sin(\pi L)}\right)^2\right)}.
\end{align*}
By evaluating the last expression on the right-hand side above numerically, we obtain that we can go up to $L<0.111$ and mantain $\|\mathcal{P}_\varepsilon\|<1$. By invoking again \cite[Theorem~1.13]{Young}, we see immediately that 
\[
\sum_{n \in \Z} \left(|f(n+\eps_n)|^2 + |f'(n+\eps_n)|^2\right) 
\]
yields an equivalent norm for $PW_{2\pi},$ as long as $\sup_n |\eps_n| < 0.111.$ 

Moreover, as $\mathcal{A}_{\eps}^{-1} : \ell^2(\N) \times \ell^2(\N) \to \ell^2(\N) \times \ell^2(\N)$ is bounded, the same argument as in the proof of Theorem \ref{th:perturbed_PW} shows that there are $\varrho_{k,n}, , \vartheta_{k,n}, \varrho'_{k,n}, \vartheta'_{k,n}$ such that 
\begin{align}\label{eq recover derivative}
f(k) &= \sum_{n \in \Z} f(n+\eps_n) \varrho_{k,n} + f'(n+\eps_n) \vartheta_{k,n}, \cr 
f'(k) &= \sum_{n \in \Z} f(n+\eps_n) \varrho'_{k,n} + f'(n+\eps_n) \vartheta'_{k,n},
\end{align}
and $\sup_n \left(\sum_{k \in \Z} \{ |\varrho_{k,n}|^2 + |\vartheta_{k,n}|^2 + |\varrho'_{k,n}|^2 + |\vartheta'_{k,n}|^2 \} \right) \lesssim 1.$ By using the adjoint $(\mathcal{A}_{\eps}^{-1})^* : \ell^2(\Z) \times \ell^2(\Z) \to \ell^2(\Z) \times \ell^2(\Z)$ 
in an analogous manner to that of the proof of Theorem \ref{th:perturbed_PW} together with \eqref{eq recover derivative} and \eqref{eq vaaler interpol},  we obtain the asserted existence of the functions $g_n, h_n \in PW_{2\pi}$ so that 
\[
f(x) = \sum_{n \in \Z} f(n + \eps_n) g_n(x) + f'(n + \eps_n) h_n(x),
\]
where the right-hand side converges absolutely, as desired. This proves the desired perturbation of Vaaler's interpolation formula, given in Theorem \ref{thm vaaler perturb}.

\section{Perturbations of Fourier interpolation on the real line}\label{perturbed radchenko viazovska}

\subsection{Improved estimates on the interpolation basis}\label{sec improve} As our goal is to obtain the perturbations of the formula
\[
f(x) = \sum_{n \ge 0} [f(\sqrt{n})a_n(x) + \widehat{f}(\sqrt{n}) \widehat{a_n}(x)]
\]
to as large as possible, we must improve the decay estimates for the interpolating functions $a_n.$ In \cite[Section~5]{RV}, the authors prove that $a_n/n^2$ is uniformly bounded 
in $n\ge 0, x \in \R.$ In order to be able to make the perturbations larger, we need to improve that result substantially, as even the refined bound $|a_n| = \mathcal{O}(n^{1/4}\log^{3/2}(1+n))$ from \cite{BRS} 
does not seem to be enough for our purposes. This first subsection is, therefore, devoted to the proof of Theorem \ref{eq improvement decay}.   

In order to prove this result we will employ the moral idea behind the characterization of Gelfand-Shilov spaces. These are spaces where, in a nutshell, both function and Fourier transform decay as fast as the negative exponential of a certain monomial. 
The following result connects these spaces with specific decay on function and Fourier side for certain Schwartz norms. See, e.g., \cite[Theorem~2.3]{CCK} for a proof.

\begin{lemma}\label{lemma gelfand shilov} Let $A,B,r,s>0$ be positive constants. The following assertions are equivalent: 

\begin{enumerate}
 \item There is $C>0$ such that 
 \[
 \sup_{x \in \R} |x^{\alpha} \varphi(x)| \le C A^{\alpha} (\alpha !)^r, \;\;\;\; \sup_{\xi \in \R} |\xi^{\beta} \widehat{\varphi}(\xi)| \le C B^{\beta} (\beta !)^s,
 \]
 for all $\alpha,\beta \in \Z;$
 \item There is $C'>0$ such that 
 \[
 |\varphi(x)| \le C' e^{-\theta |x/A|^{\frac{1}{r}}}, \;\;\;\; |\widehat{\varphi}(\xi)| \le C' e^{-\Omega |\xi/B|^{\frac{1}{s}}},
 \]
 for all $x,\xi \in \R.$ 
\end{enumerate}
\end{lemma}

We will use this result together with explicit estimates on $\{b_n^{\pm}\}_{n \ge 0},$ in the same spirit as in \cite{RV}. Indeed, let $\varepsilon \in \{\pm\}$ be a sign. In \cite{RV}, 
the authors consider the generating functions 
\[
\sum_{n = 0}^{\infty} g_n^{\eps}(z) e^{i\pi n \tau} =: K_{\eps}(\tau,z),
\]
where $g_n^{\eps}$ are weakly holomorphic modular forms of weight $3/2$ with growth and coefficient properties so that the functions 
\[
b_n^{\eps}(x) = \frac{1}{2} \int_{-1}^1 g_n^{\eps}(z) e^{i\pi x^2 z} \, \mmd z
\]
are eigenvectors of the Fourier transform associated to the eigenvalues $\eps$ satisfying that $b_n^{\pm} = a_n \pm \widehat{a_n},$ for $\{a_n\}_{n \ge 0}$ defined 
as in \ref{eq:interpolation_schwartz}. We mention, for completeness, the following result: 
\begin{theorem}[Theorem 3 in \cite{RV}]\label{thm RV functional} The following assertions hold:
\begin{align}\label{eq functional}
K_+(\tau,z) & = \frac{\theta(\tau)(1-2\lambda(\tau))\theta(z)^3 J(z)}{J(z) - J(\theta)}, \cr 
K_-(\tau,z) & = \frac{\theta(\tau)J(\tau)\theta(z)^3(1-2\lambda(z))}{J(z) - J(\theta)}, \cr 
\end{align}
where $\theta, J$ and $\lambda$ are as previously defined. Moreover, $K_{\eps}(\tau,z)$ are meromorphic functions with poles at $\tau \in \Gamma_{\theta}z,$ and the right-hand side 
of \eqref{eq functional} converges for all $\tau$ with large enough imaginary part. 
\end{theorem}

The authors then define the natural candidate for the generating function for the $\{b_n^{\eps}\}_{n \ge 0}$ to be 
\begin{equation}\label{eq gen b_n}
F_{\eps}(\tau,x) = \frac{1}{2} \int_{-1}^1 K_{\eps}(\tau,z) e^{i \pi x^2 z} \, \mmd z,
\end{equation}
which is defined, a priori, for each fixed $x \in \R$ and $\{ \tau \in \mathbb{H} \colon \forall k \in \Z, |\tau - 2k| > 1\} \supset \mathcal{D} + 2\Z,$ 
where $\mathcal{D}$ is the standard fundamental domain for $\Gamma_{\theta}.$ By Theorem \ref{thm RV functional}, there holds that, whenever $\text{Im}(\tau) >1,$ \
\begin{equation}\label{eq functional 2} 
F_{\eps}(\tau,x) = \sum_{n=0}^{\infty} b_n^{\eps}(x) e^{i\pi n \tau}.  
\end{equation}
As $F_{\eps}(\tau,x)$ admits an analytic continuation to $\mathbb{H}$ (see \cite[Proposition~2]{RV}), they are able to extend \eqref{eq functional 2} to the entire upper half space
$\mathbb{H}.$ Moreover, the following functional equations hold: 
\begin{align}
F_{\eps}(\tau,x) - F_{\eps}(\tau+2,x) &= 0, \cr 
F_{\eps}(\tau,x) + \eps(-i\tau)^{-1/2}F_{\eps}\left(\frac{-1}{\tau},x\right) &= e^{i\pi \tau x^2} + \eps (-i\tau)^{-1/2} e^{i\pi (-1/\tau)x^2}.\cr
\end{align}
The proof of Theorem \ref{eq improvement decay} follows the same essential philosophy as the proof of \cite[Theorem~4]{RV}: in order to bound each of the terms $b_n^{\pm}$, 
we bound, uniformly on $x \in \R,$ the analytic function $F_{\pm}(\tau,x)$. Relating the two bounds is achieved by the following Lemma, originally attributed to Hecke (see \cite[Lemma~1]{RV} and \cite[Lemma~2.2(ii)]{BN}):

\begin{lemma}\label{lemma fourier} Let $f: \mathbb{H} \to \C$ be a $2-$periodic analytic function admitting an absolutely convergent Fourier expansion 
\[
f(\tau) = \sum_{n \ge 0} c_n e^{i\pi n \tau}.
\]
Suppose, additionally, that for some $\alpha > 0$ it satisfies that $|f(\tau)| \le C \text{Im}(\tau)^{-\alpha},$ for $\text{Im}(\tau) < c_0.$ 
Then, for all $n > \frac{1}{c_0},$ 
\[
|c_n| \le \tilde{C} n^{\alpha}.
\]
Moreover, if $n > \frac{\alpha}{\pi c_0},$ the improved estimate 
\[
|c_n| \le C' \left(\frac{e \pi}{\alpha} \right)^{\alpha} n^{\alpha}
\]
holds.
\end{lemma}

\begin{proof}[Proof of Lemma \ref{lemma fourier}] As $f$ is analytic on $\mathbb{H}$ and its Fourier series expansion converges absolutely, an application of
Fubini's theorem gives us that 
\[
2c_n = \int_{-1 + i/n}^{1 + i/n} f(\tau) e^{- i \pi n \tau} \, \mmd \tau.
\]
The right hand side is, nonetheless, bounded in absolute value by
\begin{align*}
\int_{-1}^1 C n^{\alpha} e^{-\pi} \, \mmd t = 2Ce^{-\pi} n^{\alpha},
\end{align*}
which follows from the growth restriction on $f$ near the boundary of $\mathbb{H}$. The first assertion follows then with $\tilde{C} = 2Ce^{-\pi}.$ For the second
one, we compute instead 
\[
2c_n = \int_{-1 + i \frac{\alpha}{\pi n}}^{1 + i \frac{\alpha}{\pi n}} f(\tau) e^{-i \pi n \tau} \, \mmd \tau. 
\]
Estimating the absolute value of this integral with the given condition yields that $|c_n| \le C' \left(\frac{e \pi}{\alpha} \right)^{\alpha} n^{\alpha},$ as wished.
\end{proof}

We are now ready to prove Theorem \ref{eq improvement decay}:

\begin{proof}[Proof of Theorem \ref{eq improvement decay}] We consider the functions 
\[
F_{\eps}^k(\tau,x) := x^k F_{\eps}(\tau,x). 
\]
By Lemma \ref{lemma fourier}, if we prove that, for some $\Delta > 0,$ 
\begin{equation}\label{eq decay F^k} 
|F_{\eps}^k(\tau,x)| \le C^k (k!) \text{Im}(\tau)^{-k/2-\Delta},
\end{equation}
for all $k \ge 1,$ then we will have that 
\[
\sup_{x \in \R} |x^k b_n^{\eps}(x)| \le \tilde{C}^k n^{\Delta} n^{k/2} (k!).
\]
As $b_n^{\eps} = \eps \widehat{b_n},$ Lemma \ref{lemma gelfand shilov} then implies that each of the functions $b_n^{\eps}$ decays like 
\[
|b_n^{\eps}(x)| \lesssim n^{\Delta} e^{-\theta |x|/\sqrt{n}}, 
\]
which is the content of Theorem \ref{eq improvement decay}. Therefore, we focus on proving a suitable version of \eqref{eq decay F^k}. By the functional equation 
for $F_{\eps},$ we see that $F_{\eps}^k$ is a $2-$periodic function on $\mathbb{H}$ that satisfies the functional equation 
\begin{equation}\label{eq functional k} 
F_{\eps}^k(\tau,x) + \eps (-i\tau)^{-1/2}F_{\eps}^k(-1/\tau,x) = x^k (e^{i\pi \tau x^2} + \eps (-i\tau)^{-1/2} e^{i\pi (-1/\tau)x^2}).
\end{equation}
The strategy, in analogy to that in \cite{RV}, is of splitting in cases: if $\tau \in \mathcal{D},$ then estimates for $F_{\eps}^k$ are available \emph{directly}
by analytic methods. Otherwise, we need to use \eqref{eq functional k} to obtain the bound \eqref{eq decay F^k} for all $\tau \in \mathbb{H}.$ 

More explicitly, we have the following:

\begin{proposition}\label{prop explicit} There is a positive constants $C>0$ such that, for each $k \ge 0$ odd, the inequality 
\[
|F_{\eps}^k(\tau,x)| \le C^k (k!) (1+\text{Im}(\tau))^{-k/2}
\]
holds, whenever $\tau \in \mathcal{D}.$
\end{proposition}

This Proposition can be directly compared to \cite[Lemma~4]{RV}. In fact, it is nothing but a carefully quantified version of it. 

\begin{proof}[Proof of Proposition \ref{prop explicit}] As the proof follows thoroughly the main ideas in Lemma 4 in \cite{RV}, we will mainly focus on the points where we have to sharpen bounds. 

We see directly from the definition of $F_{\eps}^k$ that we are allowed to consider only values of $\tau \in \mathcal{D}_1 = \mathcal{D} \cap \{ \tau \in \mathbb{H} \colon \text{Re}(\tau) \in (-1,0)\}.$ 
By subsequent considerations from that reduction, we see that the bound 
\begin{equation}\label{eq first one bound}
|x^k F_{\eps}(\tau,x)| \le 10 \int_{\ell} |K_{\eps}(\tau,z)| x^k (e^{-\pi x^2 \text{Im}(\tau)} + |z|^{-1/2} e^{-\pi x^2 \text{Im}(-1/z)}) \, |\mmd z|
\end{equation}
holds, where $\ell$ is the path joining $i$ to $1$ on the upper half space, defined to be 
\begin{equation}\label{eq curve}
\ell = \left\{ w \in \mathcal{D} \colon \text{Re}(J(w)) = \frac{1}{64}, \, \text{Im}(J(w)) > 0\right\}.
\end{equation}
An explicit computation gives us that the maximal value of 
\[
x^k e^{-\pi x^2 \text{Im}(z)} 
\]
is attained at at $x = \left(\frac{k}{2 \pi \text{Im}(z)}\right)^{1/2}.$ Therefore, as any $z \in \ell$ has norm bounded from above and below by absolute 
constants, we find that there is $C>0$ so that
\begin{equation}\label{eq bound 1} 
|F_{\eps}^k(\tau,x)| \le C^{k/2} \cdot \left(\frac{k}{2\pi e}\right)^{k/2} \int_{\ell} |K_{\eps}(\tau,z)| \text{Im}(z)^{-k/2} \, |\mmd z|.
\end{equation}
We have then three regimes to consider: \\

\noindent\textit{Case 1: $|\tau - i| < 1/10.$} Notice that if we prove that the proposition holds for \emph{any} $\tau \in \mathbb{H}$ so that $|\tau - i| = \frac{1}{10},$ we can use the maximum 
modulus principle on $F_{\eps}^k$ on that circle to conclude that the proposition holds inside as well. Moreover, by the functional equation \eqref{eq functional k}, we see that the proposition holds for 
$\mathcal{A} =\{\tau \in \mathbb{H}\colon \, |\tau - i| = 1/10, |\tau| \le 1\}$ in case it holds for the image of the circle arc $\mathcal{A}$ under the action of $S.$ But a simple 
computation shows that $S\mathcal{A}$ is just another circle arc contained (up to endpoints) in $\{\tau \in \mathcal{D}_1 \colon \frac{1}{4}> |\tau - i| > \frac{1}{10}\}.$ This shows that
in order to prove the proposition for this case, it suffices to show it for the other cases.\\

\noindent\textit{Case 2: $|\tau -i| > \frac{1}{10}, \, \text{Im}(\tau) > \frac{1}{2}.$} For this case, we use the fact that $|K_{\eps}(\tau,z)| \lesssim |\theta(z)|^3 \lesssim \text{Im}(z)^{-2} e^{-\pi/ \text{Im}(z)}$ for 
$z \in \ell, \, \text{Im}(\tau) > \frac{1}{2},$ with constants independent of $\tau.$ Using this bound in \eqref{eq first one bound} yields
\[
|F_{\eps}^k(\tau,x)| \le (1+|x|^{k+2}) e^{-c|x|} \lesssim C^k \left(\frac{k+2}{e}\right)^{k+2},
\]
for some $C>0.$ Applications of Stirling's formula imply that this bound is controlled by $C_1^k (k!),$ with $C_1>0$ an absolute constant. This shows the result in this case. \\

\noindent\textit{Case 3: $|\tau -i| > \frac{1}{10}, \, \text{Im}(\tau) \le \frac{1}{2}.$}  Again, we resort to the estimates in the proof of Lemma 4 in \cite{RV}: there, the authors prove that 
\begin{align*} 
|K_+(\tau,z)| \lesssim &\text{ Im}(\tau)^{-1/2} \frac{|J(\tau)|^{3/8} |J(z)|^{5/8} \text{Im}(z)^{-3/2}}{|J(z) - J(\tau)|}, \cr 
|K_-(\tau,z)| \lesssim &\text{ Im}(\tau)^{-1/2} \frac{|J(\tau)|^{7/8} |J(z)|^{1/8} \text{Im}(z)^{-3/2}}{|J(z) - J(\tau)|}. \cr 
\end{align*}
Due to the not-so-symmetric nature of these bounds, we focus on the one for $K_{+},$ and the analysis for $K_{-},$ as well as the bounds, 
will be almost identical for the other, and thus the details will be omitted. 

Taking advantage of the explicit structure of the curve we are integrating over \eqref{eq curve}, and the fact that there is an absolute constant $C>0$ so that $\text{Im}(z)^{-1} \le C \log(1+|J(z)|)$ plus 
that $z \in \ell \iff J(z) = 1/64 + it, \, t \in \R,$ 
 
\begin{align}\label{eq last bound}
\int_{\ell} |K_+(\tau,x)| \text{Im}(z)^{-k/2} \, |\mmd z|&  \le C^{k/2} \text{Im}(\tau)^{-1/2} \int_0^{\infty} \frac{|J(\tau)|^{3/8} t^{-3/8} \log^{(k-1)/2}(1+t)}{ \sqrt{t^2 + |J(\tau)|^2}} \, \mmd t. \cr 
 &= C^{k/2} \text{Im}(\tau)^{-1/2} \int_0^{\infty} \frac{ t^{-3/8} \log^{(k-1)/2}(1+t|J(\tau)|)}{\sqrt{1+t^2}} \, \mmd t.
\end{align}
Now, the last integral in \eqref{eq last bound} can be estimated as follows: as $k-1$ is even, by using that $\log(1+ab) \le \log(1+a) + \log(1+b)$ whenever $a,b >0,$ the integral
\[
\int_0^{\infty} \frac{ t^{-3/8} \log^{(k-1)/2}(1+t|J(\tau)|)}{\sqrt{1+t^2}} \, \mmd t  
\]
is bounded by
\begin{align}\label{eq sum gamma}
\sum_{i=0}^{ \frac{k-1}{2}} {\frac{k-1}{2} \choose i} \log^i(1+|J(\tau)|) \int_0^{\infty} \frac{t^{-3/8} \log^{(k-1)/2-i}(1+t)}{\sqrt{1+t^2}} \, \mmd t.
\end{align}
Each summand above can be easily estimated. Indeed, ${(k-1)/2 \choose i} \le 2^{k/2}$ trivially, $\log^i(1+|J(\tau)|) \le C^i \text{Im}(\tau)^{-i},$ and the integrals can be explicitly bounded in terms of Gamma functions. In fact,
we first split the integrals in question as
\[
\left(\int_0^1 + \int_1^{\infty} \right) \frac{t^{-3/8} \log^{(k-1)/2 -i}(1+t)}{\sqrt{1+t^2}} \, \mmd t.
\]
For the first part, we simply bound the integrand by $t^{-3/8} \log(2)^{(k-1)/2-i},$ and this yields us a bound uniform in $k.$ 
For the second, we change variables $\log(1+t) \mapsto s$ in \eqref{eq sum gamma} above. A simple computation shows that it is bounded by 
\[
10 \int_0^{\infty} e^{-3s/8} s^{(k-1)/2-i} \, \mmd s \lesssim C^k \int_0^{\infty} e^{-r} r^{(k-1)/2 -i} \, \mmd r = C^k \Gamma\left(\frac{k-1}{2} -i+1\right).
\]
Thus, \eqref{eq sum gamma} is bounded by 
\[
C^k \text{Im}(\tau)^{(1-k)/2} \Gamma\left(\frac{k-1}{2}\right). 
\]
Putting together the estimates in \eqref{eq last bound} and \eqref{eq bound 1} and using Stirling's formula for the approximation of $\Gamma,$ we conclude that 
\[
|F_{\eps}^k(\tau,x)| \le C^k (k!) \text{Im}(\tau)^{-k/2},
\]
which was the content of the proposition.
\end{proof}

In order to finish the proof of Theorem \ref{eq improvement decay}, we first notice that $F_{\eps}^k$ is $2-$periodic, so we lose no 
generality in assuming that $\tau \in \{ z \in \mathbb{H} \colon \text{Re}(z) \in [-1,1]\} = S_1.$ If $\text{Re}(\tau) \in [-1,1],$ then we have two cases: 
\begin{enumerate}
 \item If $\tau \in \mathcal{D},$ we can use Proposition \ref{prop explicit} directly, and the decay obtained by the assertion of the Proposition remains unchanged; 
 \item If $\tau \in S_1 \backslash \mathcal{D},$ the strategy is to use \eqref{eq functional k} to reduce it to the previous case. In fact, we define the $\Gamma_{\theta}-$cocycle $\{\phi^k_A\}_{A \in \Gamma_{\theta}}$ by
\begin{align*}
\phi^k_{T^2}(\tau,x) &= 0, \cr 
\phi^k_{S}(\tau,x) &= x^k(e^{i\pi x^2 \tau} + \eps  (-i\tau)^{-1/2} e^{i\pi x^2(-1/\tau)}), \cr
\end{align*}
thogether with the cocycle relation 
\begin{equation}\label{eq cocycle}
\phi^k_{AB} = \phi^k_A + \phi^k_A|B.
\end{equation}
For a fixed $\tau \in S_1 \setminus \mathcal{D},$ we associate $\tau' \in \mathcal{D}$ through the following process: let 
\begin{align}
\begin{cases} 
\gamma_0 &= \tau, \cr
\gamma_i &= -\frac{1}{\gamma_{i-1}} - 2n_i,\cr
\end{cases}
\end{align}
where $n_i = \left\lfloor \frac{(-1/\gamma_{i-1})+ 1}{2} \right\rfloor.$ We define $m=m(\tau)$ to be the smallest positive integer so that $\gamma_m \in \mathcal{D}.$ In this case, we let $\gamma_{m(\tau)} =: \tau'.$  In other words, we have that the sequence 
\begin{align}
\begin{cases} 
\tau_0 &= \tau', \cr 
\tau_{i+1} &= -\frac{1}{\tau_i} + 2n_i\cr
\end{cases}
\end{align}
satisfies the hypotheses of Lemma 3 in \cite{RV}. We therefore have that $|\tau_j| > 1,\, \text{Im}(\tau_j)$ is nonincreasing and $\text{Im}(\tau_j) \le \frac{1}{2j-1}.$ An inductive procedure shows us that 
\[
\gamma_{m-i} = -\frac{1}{\tau_i}.
\]
In particular, the sequence $\{\tau_i\}_{i \ge 0}$ is in fact finite, with at most $m(\tau)$ terms. This implies that
\begin{equation}\label{eq bound steps}
m+1 \le 4m-2 \le 2\text{Im}(\tau)^{-1}.
\end{equation}
We will use \eqref{eq bound steps} in the following computation with the cocycle condition. We write $\tau' = A\tau,$ where $A \in \Gamma_{\theta}$ is of the form 
\[
A = ST^{2n_m} S T^{2n_{m-1}}S \cdots T^{2n_1}S.
\]
As $\{\phi^k_A\}_{A \in \Gamma_{\theta}}$ satisfies the cocycle condition \eqref{eq cocycle}, the proof of Lemma 3 in \cite{RV} gives us that 
\[
\text{Im}(\tau')^{1/4} |\phi^k_A(\tau')| \le \sum_{j=1}^m \text{Im}(\tau_j)^{1/4} |\phi^k_S(\tau_j)|.
\]
By the definition of $\phi^k_S,$ we see that 
\begin{equation}\label{eq cocycle bound}
|\phi^k_S(\tau_j,x)| \le C \Gamma\left(\frac{k+1}{2}\right)(\text{Im}(\tau_j)^{-k/2} + |\tau_j|^{-1/2} \text{Im}(-1/\tau_j)^{-k/2}).
\end{equation}
As $\gamma_{m-i} = -\frac{1}{\tau_i} = \tau_{i+1} - 2n_i,$ $|\tau_j| > 1,$ and the sequence $\text{Im}(\tau_j)$ is nonincreasing, the right-hand side of \eqref{eq cocycle bound} 
is bounded from above by $C \cdot \Gamma((k+1)/2) \text{Im}(\tau)^{-k/2}.$ From \eqref{eq bound steps}, it follows that 
\[
|\phi^k_A(\tau')| \text{Im}(\tau')^{1/4} \le C \Gamma\left(\frac{k+1}{2}\right) \text{Im}(\tau)^{-k/2} \left(\sum_{j = 1}^m \text{Im}(\tau_j)^{1/4}\right).
\]
If we use the aforementioned facts about $\text{Im}(\tau_j),$ we will see that, in fact, 
\begin{equation}\label{eq final cocycle bound}
|\phi^k_A(\tau')| \text{Im}(\tau')^{1/4} \le C \Gamma\left(\frac{k+1}{2} \right) \text{Im}(\tau)^{-k/2} m(\tau)^{3/4}. 
\end{equation}
Now, using the functional equation for $F_{\eps}^k$ implies
\[
F_{\eps}^k - (F_{\eps}^k)|A = \phi^k_A, 
\]
which then gives us 
\[
|F^k_{\eps}(\tau,x)| |\text{Im}(\tau)|^{1/4} \le |\text{Im}(\tau')| ^{1/4}|F^k_{\eps}(\tau',x)| + |\phi^k_A(\tau',x)||\text{Im}(\tau')|^{1/4}.
\]
Denoting $\text{Im}(\tau') =: I(\tau)$ and using Proposition \ref{prop explicit} and \eqref{eq final cocycle bound} to estimate this expression, it follows that
\begin{equation}\label{eq final modular bound} 
|F^k_{\eps}(\tau,x)| \le \text{Im}(\tau)^{-k/2-\frac{1}{4}} \left( C^k (k!)\cdot I(\tau)^{1/4} + \Gamma((k+1)/2) m(\tau)^{3/4} \right).
\end{equation}
\end{enumerate}

In order to estimate \eqref{eq final modular bound}, we must resort not only to Lemma \ref{lemma fourier} and its proof, but also to the following estimate of the average values 
of $m(\tau)$ and $I(\tau)$, recently available by the work of Bondarenko, Radchenko and Seip. We refer the reader to Propositions 6.6 and 6.7 in \cite{BRS} for a proof. 

\begin{lemma}\label{lemma average bound} Whenever $y \in (0,1/2),$ we have
\[
\int_{-1}^1 I(x+iy)^{1/4} \lesssim 1
\]
and 
\[
\int_{-1}^1 m(x+iy)^{3/4} \lesssim \log^{3/2}(1+y^{-1}).
\]
\end{lemma}
An application of Lemma \ref{lemma average bound} together with the bound \eqref{eq final modular bound} to the proof of the first bound in Lemma \ref{lemma fourier} implies
\begin{equation}\label{eq double bound}
\sup_{x \in \R} |x^k b_n^{\pm}(x)| \lesssim C^k n^{1/4} n^{k/2} \log^{3/2}(1+n) (k!)
\end{equation}
for $n > \frac{1}{c_0}, k \ge 1.$ Also, in case $n \ge \frac{k}{\pi c_0},$ the sharper bound 
\begin{equation}\label{eq double bound2}
\sup_{x \in \R} |x^k b_n^{\pm}(x)| \lesssim (C')^k n^{1/4} n^{k/2} \log^{3/2}(1+n) (k!)^{1/2} 
\end{equation}
holds instead. We now employ then the main idea of proof of Lemma \ref{lemma gelfand shilov}: we seek to optimize in $k >0$. 

Indeed, let us start by optimizing \eqref{eq double bound}. We postpone the discussion on the improved bound \eqref{eq double bound2} to a later remark. 

Notice that we may assume $|x| \ge C' \sqrt{n},$ as for if $|x| < C'\sqrt{n},$ the bound \eqref{eq double bound} with $k=0$ gives us already the result, as $1 \lesssim_c e^{-c|x|/\sqrt{n}}.$ 
If we then set $k = \frac{|x|}{C'\sqrt{n}},$ where $C'>0$ will be a fixed positive constant, whose exact value shall be determined later, we have that 
\begin{align*}
|b_n^{\pm}(x)| &\lesssim n^{1/4} \log^{3/2}(1+n) \cdot \text{exp}( k \log(Cn^{1/2}) + k \log (k) - k \log|x| ) \cr
\end{align*}
The exponential term above is
$$\text{exp}\left( \frac{|x|}{C'\sqrt{n}} \log(Cn^{1/2}) + \frac{|x|}{C'\sqrt{n}} (\log(|x|) - \log(C'\sqrt{n})) - \frac{|x|}{C'\sqrt{n}} \log|x|\right)$$
\[
 = \text{exp}\left(\frac{|x|}{C'\sqrt{n}} \log \left(\frac{C}{C'}\right) \right).
\]
We only need to set $C' \ge 2C$ above, and this quantity will grow like $\text{exp}(-c|x|/\sqrt{n}).$ This finishes the first assertion in Theorem \ref{eq improvement decay}.

For the second one, we notice that the proof above adapts in many instances. Indeed, if we shift our attentionto the function $\partial_x F^k_{\eps}(\tau,x)$ instead, we will see that, 
in an almost identical fashion to that of the proof of Proposition \ref{prop explicit}, we are able to prove that, for all $\tau \in \mathcal{D},$
\[
|\partial_x F_{\eps}^k(\tau,x)| \lesssim C^k (k!) \text{Im}(\tau)^{-\frac{k+1}{2}}.
\]
On the other hand, the partial derivative $\partial_x$ of the cocycle $\{\phi^k_A\}_{A \in \Gamma_{\theta}}$ is itself a cocycle with respect to the same slash operator. Moreover, for $A = S,$ the following formula holds: 
\[
\partial_x \phi^k_S (\tau,x) = (2\pi i) x^{k+1} \left( \tau e^{\pi i x^2 \tau} + i \eps (-i\tau)^{-3/2} e^{\pi i x^2 (-1/\tau)}\right).
\]
In that case, using the notation
from above for the elements $\tau', \tau_j \in \mathbb{H}$ associated to $\tau \in \mathbb{H} \cap \{|z| \le 1\},$ we see that
\[
\text{Im}(\tau')^{1/4} |\partial_x \phi^k_A(\tau')| \le \text{Im}(\tau')^{1/4} |\partial_x \phi^k_S(\tau')| + \sum_{j = 1}^m \text{Im}(\tau_j)^{1/4} |\partial_x \phi^k_A(\tau_j)|.
\]
For $j \in \{0,1,2,\dots,m\},$ the definition of our new cocycle implies 
\[
|\partial_x \phi^k_S(\tau_j,x)| \lesssim \Gamma\left(\frac{k+3}{2}\right) \left(|\tau_j| \text{Im}(\tau_j)^{-\frac{k+1}{2}} + |\tau_j|^{-3/2} \text{Im}(\tau_{j+1})^{-\frac{k+1}{2}}\right) 
\]
\[
\le \Gamma\left(\frac{k+3}{2}\right) \text{Im}(\tau)^{-\frac{k+1}{2}}.
\]
This follows as before from the fact that $\text{Im}(\tau_{j+1}) = \frac{\text{Im}(\tau_j)}{|\tau_j|^2} \ge \text{Im}(\tau)$ and that $|\tau_j| > 1.$ Analyzing the functional equations for 
$\partial_x F^k_{\eps}(\tau,x)$ in the same way as before readily gives that
\[
|\partial_x F^k_{\eps}(\tau,x)| \le C^k \text{Im}(\tau)^{-\frac{k+1}{2} - \frac{1}{4}} (k!) \left( I(\tau)^{1/4} + m(\tau)^{3/4}\right).
\]
Lemma \ref{lemma average bound} and the considerations employed for $F^k_{\eps}$ apply almost verbatim here, and thus we conclude that 
\[
|(b_n^{\pm})'(x)| \lesssim n^{3/4} \log^{3/2}(1+n) e^{-c|x|/\sqrt{n}},
\]
as wished.

\end{proof}

As a consequence of Theorem \ref{eq improvement decay}, we are able to establish the following bound for the interpolation basis taking account both decay and zeros.

\begin{corollary} Let $\{a_n\}$ be the interpolation sequence of functions from \eqref{eq:interpolation_schwartz}. Then there is $c>0$ so that 
\[
|a_n(x)| \lesssim n^{3/4} \log^{3/2}(1+n)\text{dist}(|x|,\sqrt{\N}) e^{-c\frac{|x|}{\sqrt{n}}},
\]
for all positive integers $n \in \N.$
\end{corollary}

\begin{proof} We simply use the fundamental theorem of calculus to the $a_n:$ without loss of generality, we suppose $x >0.$ We then have:
\begin{align*}
|a_n(x)| & = |a_n(x) -a_n(\sqrt{m}) +\delta_{n,m}|  \le \int_{\sqrt{m}}^x |a_n'(x)| \, \mmd x + \delta_{n,m} \cr 
 & \le n^{3/4} \log^{3/2}(1+n) \text{dist}(x,\sqrt{\N}) e^{-c\frac{|x|}{\sqrt{n}}} + \delta_{m,n} \cr 
 & \lesssim  n^{3/4} \log^{3/2}(1+n) \text{dist}(x,\sqrt{\N}) e^{-c\frac{|x|}{\sqrt{n}}}, \cr
\end{align*}
as the $\delta_{m,n}$ factor is only one if $|x| \in [\sqrt{n},\sqrt{n+1}),$ where $1 \lesssim e^{-c|x|/\sqrt{n}}.$ 
\end{proof}

\noindent\textbf{Remark.} Although the exponential bound $n^{1/4} \log^{3/2}(1+n) e^{-c|x|/\sqrt{n}}$ sufficies for our purposes, below we sketch how to deduce a slightly improved decay for the interpolation basis $\{a_n\}_{n \ge 0}.$

We again wish to optimize \eqref{eq double bound2}. If we set $k = \frac{|x|^2}{C'n},$ where $C'>0$ will be chosen soon, we have that 
\begin{align*}
|b_n^{\pm}(x)| \lesssim n^{1/4} \log^{3/2}(1+n) \cdot \text{exp}( k \log(Cn^{1/2}) + k \log (k^{1/2}) - k \log|x| ).
\end{align*}
This bound holds as long as $\pi n \gtrsim k \ge 1.$ If instead $k<1,$ that means, $|x| \le \sqrt{C'} \sqrt{n},$ we use the bound in either \eqref{eq double bound} or \eqref{eq double bound2} for
$k = 0,$ which yields $|b_n^{\pm}(x)| \lesssim n^{1/4} \log^{3/2}(1+n) \lesssim n^{1/4} \log^{3/2}(1+n) e^{-c|x|^2/n},$ for $c >0$. 

On the other hand, in case $k>1,$ the first exponential term above becomes 
$$\text{exp}\left( \frac{|x|^2}{C'n} \log(Cn^{1/2}) + \frac{|x|^2}{C'n} (\log(|x|) - \log(\sqrt{C'n})) - \frac{|x|^2}{C'n} \log|x|\right)$$
\[
 = \text{exp}\left(\frac{|x|^2}{C'n} \log \left(\frac{C}{\sqrt{C'}}\right) \right).
\]
We only need to set $C' \ge (2C)^2$ above, and this quantity will grow like $\text{exp}(-c|x|^2/n).$ 

For the remaining $|x| > \sqrt{C'} n$ case, we need to refine the analysis of the proof of Lemma \ref{lemma fourier} and Theorem \ref{eq improvement decay}. Indeed, it is easy to see that if 
$n \in (2^{-j} \alpha, 2^{1-j} \alpha), \, j \ge 1,$ then evaluating the Fourier coefficients of a 2-periodic function $f: \mathbb{H} \to \C$ such that $|f(\tau)| \lesssim \text{Im}(\tau)^{-\alpha}\left(I(\tau)^{1/4} + m(\tau)^{3/4}\right)$ 
for $\text{Im}(\tau) \le 1$ as $2c_n = \int_{-1 + i \frac{\alpha}{2^j \pi n}}^{1+i \frac{\alpha}{2^j \pi n}} f(\tau) e^{-\pi i n \tau} \, \mmd \tau$ 
implies 
\[
|c_n| \lesssim \left(\frac{2^j \pi e^{1/2^j}}{\alpha}\right)^{\alpha} n^{\alpha} \log^{3/2}(1+n).
\]
Using this new bound in \eqref{eq final modular bound}, we obtain that, when $n \in (2^{-j-1} k, 2^{-j}k),$ 
\begin{align*}
|b_n^{\pm}(x)| &\lesssim n^{1/4} \log^{3/2}(1+n)  \cdot \exp\left(k \left( j/2 +  \log(C\sqrt{n}) + \log(k^{1/2})- \log|x|\right)\right).
\end{align*}
This suggests that we take $k = \frac{|x|^2}{C' 2^j n},$ which is admissible to the condition $n \in (2^{-j-1} k, 2^{-j}k)$ if $|x| \sim \sqrt{C'} 2^j n.$ A similar computation to the ones above 
implies that 
\[
|b_n^{\pm}(x)| \lesssim n^{1/4} \log^{3/2}(1+n) \exp\left(-c\frac{|x|^2}{2^j n}\right) \lesssim n^{1/4} \log^{3/2}(1+n) \exp(-c'|x|),
\]
whenever $C' \gg C.$ The next corollary then follows as a natural consequence. 

\begin{corollary}\label{thm bound best} Let $a_n: \R \to \R$ be the interpolating functions in the Radchenko--Viazovska interpolation formula. Then there are $c,C>0$ so that 
\[
|a_n(x)| \lesssim n^{1/4} \log^{3/2}(1+n) \left( e^{-c|x|^2/n} 1_{|x| < C n} + e^{-c|x|} 1_{|x| > Cn} \right),
\]
for each $n \ge 1.$
\end{corollary}

Indeed, the application of Lemma \ref{lemma fourier} requires that we take $n \ge C,$ for some $C>0$ some absolute constant. In order to prove such a result 
for $n \lesssim 1,$ we may simply use the definition of $b_n^{\pm}$ as a Laplace transform of a the weakly holomorphic modular form $g_n^{\pm}.$ Indeed, in order to extend Corollary \ref{thm bound best} 
to $n=0,$ we write 
\[
a_0(x) = \widehat{a_0}(x) = \frac{1}{4} \int_{-1}^1 \theta(z)^3 \, e^{\pi i x^2 z} \, \mmd z. 
\]
In order to prove that $a_0$ decays exponentially, we employ a similar technique to that of \cite[Proposition~1]{RV}. Indeed, we have
\[
|\theta(z)|^3 \lesssim \text{Im}(z)^{-2} \, e^{-\pi/\text{Im}(z)} \, \text{ for } z \to \pm 1
\]
and moreover that $|\theta(z)| \lesssim 1$ whenever $z \in \mathbb{H}, |z| = 1.$ We also suppose without loss of generality that $x > 0.$ This implies that, for $\delta >0,$
\[
|a_0(x)| \lesssim \int_0^{\delta} \frac{e^{-1/(2t)}}{t^2} \, \mmd t + e^{-\pi x^2 \delta} \lesssim e^{-\frac{1}{2\delta}} + e^{-\pi x^2 \delta}.
\]
We then choose, for $x \gg 1,\, \delta = \frac{1}{\sqrt{2\pi}x}.$ This implies that $|a_0(x)| \lesssim e^{-\sqrt{\frac{\pi}{2}} x},$ which is the desired bound. For other bounded 
values of $n$ such a proof can be easily adapted. 

\subsection{Proof of the main result}\label{sec main result}

Let 
$$\ell^2_s(\N)=\{(a_n)_n\in\ell^2(\N)\,:\,(n^s a_n)_n\in\ell^2(\N)\}.$$
Let $I:\ell^2_s(\N) \times \ell^2_s(\N) \to \ell^2_s(\N)\times \ell^2_s(\N)$ denote the identity operator. 
Recall the Radchenko-Viazovska interpolation result: for $f \in \mathcal{S}_{even}(\R)$ a real function, 
\begin{equation}\label{eq radchenko viazovska}
 f(x) = \sum_{n \ge 0} (f(\sqrt{n}) a_n(x) + \widehat{f}(\sqrt{n}) \widehat{a_n}(x)),
\end{equation}
where $a_n: \R \to \R$ is a sequence of interpolating functions independent of the Schwartz function $f.$ In particular, 
\[
f(\sqrt{k}) = \sum_{n \ge 0} (f(\sqrt{n}) a_n(\sqrt{k}) + \widehat{f}(\sqrt{n}) \widehat{a_n}(\sqrt{k})).
\]
In fact, for any pair of sequences $(\{x_i\}_i,\{y_i\}_i)$ decaying sufficiently fast and satisfying the Poisson summation formula 
\begin{equation}\label{eq poisson}
\sum_{n \in \Z} x_{n^2} = \sum_{n \in \Z} y_{n^2},
\end{equation}
the function 
\begin{equation}\label{eq g function}
\mathfrak{G}(t) = \sum_{n \ge 0} (x_n a_n(t) + y_n \widehat{a_n}(t))
\end{equation}
is well-defined and satisfies that $\mathfrak{G}(\sqrt{k}) = x_k, \widehat{\mathfrak{G}}(\sqrt{k}) = y_k.$ In fact, let $(\{x_i\}_i,\{y_i\}_i) \in \ell^2_s(\N) \times \ell^2_s(\N)$ for $s > 0$ sufficiently large. The operator 
$$T:\ell^2_s(\N) \times \ell^2_s(\N) \to \ell^2_s(\N) \times \ell^2_s(\N)$$
given by $T = (T^1,T^2),$ where 
\begin{align*}
T^1(\{x_i\},\{y_i\})_k &= \sum_{n \ge 0} (x_n a_n(\sqrt{k}) + y_n \widehat{a_n}(\sqrt{k})), \cr 
T^2(\{x_i\},\{y_i\})_k &= T^1(\{y_i\},\{x_i\})_k, \cr 
\end{align*} 
has an explicit form: indeed, for $k \ge 1,$ we have 
\[
T^1(\{x_i\},\{y_i\})_k = x_k, \, T^2(\{x_i\},\{y_i\}) = y_k.
\]
For $k=0,$ we have
\[
T^1(\{x_i\},\{y_i\})_0 = \frac{x_0 + y_0}{2} - \sum_{n \ge 1} x_{n^2} + \sum_{n \ge 1} y_{n^2},
\]
\[
T^2(\{x_i\},\{y_i\})_0 = \frac{x_0+y_0}{2} - \sum_{n \ge 1} y_{n^2} + \sum_{n \ge 1} x_{n^2}.
\]
In particular, it is then easy to see that $T = I$ whenever $(\{x_i\}_i,\{y_i\}_i)$ satisfy the Poisson relation \eqref{eq poisson}. Inspired by this fact, we define the perturbed operator associated to a sequence 
$\eps_k > 0, k \in \N,$ to be 
\[
\tilde{T} \text{ defined on } \ell^2_s(\N) \times \ell^2_s(\N),
\]
where $\tilde{T} =(\tilde{T}^1,\tilde{T}^2),$ with 
\begin{align*}
\tilde{T}^1(\{x_i\},\{y_i\})_k &= \sum_{n \ge 0} (x_n a_n(\sqrt{k+\eps_k}) + y_n \widehat{a_n}(\sqrt{k+\eps_k})), \cr 
\tilde{T}^2(\{x_i\},\{y_i\})_k &= \tilde{T}^1(\{y_i\},\{x_i\})_k, \cr 
\end{align*} 
for $k\ge 1,$ and $\tilde{T}^1(\{x_i\},\{y_i\})_0 = x_0, \tilde{T}^2(\{x_i\},\{y_i\})_0 = y_0.$ 
A fundamental fact we will need for our proof is that this operator is \emph{bounded} from $\ell^2_s(\N) \times \ell^2_s(\N) \to \ell^2_s(\N) \times \ell^2_s(\N).$ One way to see this will be provided in the proof of our main theorem,
by showing that the operator norm $\| I - \tilde{T}\|_{\ell^2_s(\N) \times \ell^2_s(\N)) \to \ell^2_s(\N) \times \ell^2_s(\N))} < +\infty.$ This is, incidentally, our main device to prove our result: if 
\[
\| I - \tilde{T}\|_{\ell^2_s(\N) \times \ell^2_s(\N)) \to \ell^2_s(\N) \times \ell^2_s(\N))} < 1,
\]
then $\tilde{T}$ is an invertible operator defined on $\ell^2_s(\N) \times \ell^2_s(\N).$ Therefore, its inverse 
\[
\tilde{T}^{-1} :\ell^2_s(\N) \times \ell^2_s(\N) \to \ell^2_s(\N) \times \ell^2_s(\N)
\]
is well-defined and bounded. In particular, for $f \in \mathcal{S}_{even}(\R)$ real, given the lists of values 
\[
f(0),f(\sqrt{1+\eps_1}),f(\sqrt{2+\eps_2}),\cdots,
\]
\[
\widehat{f}(0), \widehat{f}(\sqrt{1+\eps_1}),\widehat{f}(\sqrt{2+\eps_2}),\cdots,
\]
there is a unique pair $(\{x_i\}_i,\{y_i\}_i) \in \ell^2_s(\N) \times \ell^2_s(\N)$ so that 
$$\tilde{T}(\{x_i\},\{y_i\}) = (\{f(\sqrt{k+\eps_k})\}_k,\{\widehat{f}(\sqrt{k+\eps_k})\}_k).$$
But we also know that 
\[
\tilde{T}(\{f(\sqrt{i})\}_i,\{\widehat{f}(\sqrt{i})\}_i) = T(\{f(\sqrt{i})\}_i,\{\widehat{f}(\sqrt{i})\}_i)=\{f(\sqrt{k+\eps_k})\}_k,\{\widehat{f}(\sqrt{k+\eps_k})\}_k). 
\]
This implies $x_j = f(\sqrt{j}), \, y_j = \widehat{f}(\sqrt{j}).$ By writing the $k-$th entry of the inverse of $\tilde{T}$ as 
\[
\tilde{T}^{-1}(\{w_i\},\{z_i\})_k = \sum_{j \ge 0} (\gamma_{j,k} w_j + \widehat{\gamma}_{j,k} z_j), 
\]
for two sequences $\{\gamma_{j,k}\}_{j,k\ge 0}, \, \{\widehat{\gamma}_{j,k}\}_{j,k \ge 0}$ so that 
$|\gamma_{j,k}| + |\widehat{\gamma}_{j,k}|\lesssim (j/k)^s,$ we must have 
\begin{equation}\label{eq inverse}
f(\sqrt{k}) = \sum_{j \ge 0} (\gamma_{j,k} f(\sqrt{j+\eps_j}) + \widehat{\gamma}_{j,k} \widehat{f}(\sqrt{j+\eps_j})).
\end{equation}
This implies, by \eqref{eq:interpolation_schwartz}, that we can recover $f$ from its values and those of its Fourier transform at $\sqrt{k+\eps_k}.$ Moreover, as the adjoint 
of $\tilde{T}^{-1}$ is also bounded from $\ell^2_s(\N) \times \ell^2_s(\N)$ to itself, we conclude that, for $s \gg 1$ sufficiently large and $f,\widehat{f}$ both being 
$\mathcal{O}((1+|x|)^{-10s}),$ we can use Fubini's theorem in \eqref{eq:interpolation_schwartz} together with \eqref{eq inverse}. This proves the existence of two sequences of functions $\{\theta_j\}_{j\ge0}, \{\eta_j\}_{j\ge0}$
so that 
\[
|\theta_j(x)| + |\eta_j(x)| + |\widehat{\theta}_j(x)| + |\widehat{\eta}_j(x)| \lesssim (1+j)^s (1+|x|)^{-10}
\]
and
\[
f(x) = \sum_{j \ge 0} \left( f(\sqrt{j+\eps_j}) \theta_j(x) + \widehat{f}(\sqrt{j+\eps_j}) \eta_j(x)\right).
\]
Thus, we focus on the proof of the invertibility of $\tilde{T},$  for $s > 0$ suitably chosen.

\begin{proof}[Proof of invertibility of $\tilde{T}$] We use, for this part, the Schur test. That is, define the infinite matrices $A = \{A_{ij}\}_{i,j > 0}$ and $\widehat{A} = \{\widehat{A}_{ij}\}_{i,j > 0}$ by 
\begin{align*}
A_{ij} &= (a_j(\sqrt{i+\eps_i}) - \delta_{ij})\times (i/j)^s, \cr 
\widehat{A}_{ij} &= \widehat{a_j}(\sqrt{i+\eps_i}) (i/j)^s. \cr
\end{align*}
For a given vector $(x,y) \in \ell^2(\N) \times \ell^2(\N),$ we write then 
\[
B(x,y) = (A\cdot x + \widehat{A} \cdot y , A \cdot y + \widehat{A} \cdot x),
\]
or, in matrix notation, 
\[
B = \begin{pmatrix}
                 A &  \widehat{A} \\ \widehat{A} & A
                \end{pmatrix}.
\]
Notice that the operator norm of $\tilde{T} - I$ acting on $\ell^2_s(\N) \times \ell^2_s(\N)$ is, by virtue of our definitions, the \emph{same} as the operator norm of 
$B$ acting on $\ell^2(\N) \times \ell^2(\N).$ Therefore, it will suffice to impose bounds on this latter quantity. 

By Schur's test, it suffices to find $\alpha, \beta >0$ and positive sequences $\{p_i\}_{i \ge 0}, \{q_i\}_{i \ge 0}$ so that the following inequalities hold: 
\begin{align}\label{eq schur1}
\sum_{j > 0} (i/j)^s \times &\left[ |a_j(\sqrt{i+\eps_i}) - \delta_{ij}|p_j + |\widehat{a_j}(\sqrt{i+\eps_i})| q_j \right] \le \alpha p_i, \cr 
\sum_{j > 0} (i/j)^s \times &\left[ |a_j(\sqrt{i+\eps_i}) - \delta_{ij}|q_j + |\widehat{a_j}(\sqrt{i+\eps_i})| p_j \right] \le \alpha q_i, \cr 
\sum_{i > 0} (i/j)^s \times &\left[ |a_j(\sqrt{i+\eps_i}) - \delta_{ij}|p_i + |\widehat{a_j}(\sqrt{i+\eps_i})| q_i \right] \le \beta p_j, \cr 
\sum_{i > 0} (i/j)^s \times &\left[ |a_j(\sqrt{i+\eps_i}) - \delta_{ij}|q_i + |\widehat{a_j}(\sqrt{i+\eps_i})| p_i \right] \le \beta q_j. \cr 
\end{align}
Now, we make the Ansatz that, for all $i >0,\, p_i = q_i = i^{\theta},$ for some real number $\theta \in \R.$ By making use of Theorem \ref{eq improvement decay}, we know that 
\[
|a_j(\sqrt{i+\eps_i}) - \delta_{ij}| + |\widehat{a_j}(\sqrt{i+\eps_i})| \lesssim \frac{\eps_i}{\sqrt{i}} j^{3/4} e^{-c\sqrt{i/j}}.
\]
Therefore, \eqref{eq schur1} reduces to verifying 
\begin{align}\label{eq schur2}
\sum_{j > 0} (i/j)^s \times j^{\theta} & \times \frac{\eps_i}{\sqrt{i}} j^{3/4} e^{-c\sqrt{i/j}} \le \alpha i^{\theta}, \cr 
\sum_{i > 0} (i/j)^s \times i^{\theta} & \times \frac{\eps_i}{\sqrt{i}} j^{3/4} e^{-c\sqrt{i/j}} \le \beta j^{\theta}. \cr 
\end{align}
\noindent\textit{Estimate of the first term in \eqref{eq schur2}.} For this term, we rewrite it as
\[
i^{s-1/2}\times \eps_i \left( \sum_{j > 0} j^{3/4 - s} e^{-c\sqrt{i/j}} j^{\theta}\right). 
\]
In order to estimate this last sum, we break it into $j < i^{1/3}$ and $j > i^{1/3}$ contributions. Therefore, 
\begin{equation}\label{eq sum bound} 
\sum_{j > 0} j^{3/4 - s} e^{-c\sqrt{i/j}} j^{\theta} \lesssim i^{1/3} i^{\max(3/4-s+\theta,0)} e^{-ci^{1/3}} + \sum_{j > i^{1/3}} j^{3/4 - s} e^{-c\sqrt{i/j}} j^{\theta}. 
\end{equation}
Because of the presence of the exponential, the first term is always bounded by an absolute constant times $i^{\theta},$ so we treat it as negligible. For the second term, notice that 
the summand is bounded by a constant times $\int_j^{j+1} x^{3/4-s+\theta} e^{-c\sqrt{i/x}} \mmd x.$ Indeed, the ratio between both is bounded by 
\begin{align*}
\int_j^{j+1} (x/j)^{3/4-s+\theta} e^{c \left(\sqrt{i/j} - \sqrt{i/x}\right)} \, \mmd x & \le 2^{3/4-s+\theta} \sup_{x \in [j,j+1)} e^{c\frac{\sqrt{i}}{\sqrt{jx}} (\sqrt{x} - \sqrt{j})}\cr
										       & \le 2^{3/4 -s+\theta} e^{c'\frac{\sqrt{i}}{\sqrt{j^3}}} \lesssim_{s,\theta} 1,\cr
\end{align*}
as $j > i^{1/3}.$ Thus, we obtain that the second term on the right-hand side of \eqref{eq sum bound} is bounded by 
\begin{align*}
\int_{i^{1/3}}^{\infty} x^{3/4-s+\theta} e^{-c\sqrt{i/x}} \, \mmd x & = \int_0^{i^{-1/3}} (1+1/y)^{3/4-s+\theta} y^{-2} e^{-c\sqrt{iy}} \, \mmd y \cr
\lesssim_{s,\theta} \int_0^{i^{-1/3}}  y^{-11/4 + s-\theta} e^{-c\sqrt{iy}} \, \mmd y & = i^{7/4-s+\theta} \int_0^{i^{2/3}} y^{-11/4+s-\theta} e^{-c\sqrt{y}} \, \mmd y \cr
\lesssim_{s,\theta} i^{7/4-s+\theta},
\end{align*}
as long as $-11/4+s-\theta > -1,$ that is, $\theta < s-7/4.$ Thus, the first term in \eqref{eq schur2} is bounded under such a condition by 
\[
C_{s,\theta} \eps_i i^{s-\frac{1}{2}} i^{\frac{7}{4} - s + \theta} = i^{\frac{5}{4}+\theta} \eps_i. 
\]
In order for this last quantity to be less than $\alpha i^{\theta},$ we must have $\eps_i \lesssim_{s,\theta} \alpha i^{-\frac{5}{4}}.$ We will assume that we have this bound while estimating the second term. \\

\noindent\textit{Estimate for the second term in \eqref{eq schur2}.} For the second term, the strategy is similar, only now the estimates become somewhat simpler by the arithmetic of the 
bounds given by Theorem \ref{eq improvement decay}. Indeed, the second term in \eqref{eq schur2} is bounded by 
\[
c_{s,\theta} j^{\frac{3}{4} - s} \left(\sum_{i>0}  i^{s+\theta-\frac{7}{4}}  e^{-c\sqrt{i/j}} \right). 
\]
Similarly as before, each summand above is bounded by $\int_i^{i+1} x^{s+\theta-\frac{7}{4}} e^{-c\sqrt{x/j}} \, \mmd x.$ Thus, the expression within the parenthesis above is bounded by 
\[
\int_1^{\infty} x^{s+\theta-\frac{7}{4}} \, e^{-c\sqrt{x/j}} \, \mmd x \lesssim_{s,\theta} j^{s+\theta-\frac{3}{4}} \int_0^{\infty} x^{s+\theta-\frac{7}{4}} e^{-c\sqrt{x}} \, \mmd x.
\]
This last integral converges given that $s+\theta-\frac{7}{4} > - 1 \iff s+\theta > \frac{3}{4}.$ In the end, we obtain that the second term in \eqref{eq schur2} is bounded by $c_{s,\theta} j^{\theta}$
if these conditions on $s,\theta$ hold. 

Finally, we gather these two estimates to get that, if $s-\theta > \frac{7}{4}, s+\theta>\frac{3}{4}$ and if $\eps_i < \gamma i^{-\frac{5}{4}}$ for $\gamma > 0$ sufficiently small, 
then both terms of \eqref{eq schur2} are bounded by small constants times $i^\theta$ and $j^\theta.$ Notice that picking $s = 10$ and $\theta>0$ sufficiently small yields that both conditions 
above hold true, and thus the result follows from Schur's test, as previously indicated.
\end{proof}

As mentioned in the beginning of this manuscript, the usage of Schur's test here was instrumental in order to expand the range of our perturbations. In fact, in \S \ref{sec origin}, se employ the 
Hilbert--Schmidt test successfully to our operator $\tilde{T}$ and obtain that, as long as there is $\delta>0$ such that $\eps_i \lesssim i^{-\frac{5}{4}-\delta},$ then $\tilde{T}$ is bounded on
$\ell^2_s(\N) \times \ell^2_s(\N),$ for $s$ sufficiently large, but we seem to be unable to include $5/4$ in our considerations with the Hilbert--Schmidt method. 

On the other hand, we will see in that subsection that the Hilbert--Schmidt method provides us with a way to suitably perturb the origin, a feature we could not obtain with Schur's test. 

\section{Applications of the main results and techniques}\label{sec applications}

\subsection{Interpolation formulae perturbing the origin}\label{sec origin} In the main results of this manuscript, the only interpolation node that remains unchanged in every scenario is $0.$ One of the reasons for that 
is aesthetic: we are concerned mainly with even functions here, so the origin keeps a sense of symmetry. The other main reason is technical: we recall that the operator 
$$T:\ell^2_s(\N) \times \ell^2_s(\N) \to \ell^2_s(\N) \times \ell^2_s(\N)$$
given by $T = (T^1,T^2),$ where 
\begin{align*}
T^1(\{x_i\},\{y_i\})_k &= \sum_{n \ge 0} (x_n a_n(\sqrt{k}) + y_n \widehat{a_n}(\sqrt{k})), \cr 
T^2(\{x_i\},\{y_i\})_k &= T^1(\{y_i\},\{x_i\})_k, \cr 
\end{align*} 
for $k \ge 0,$ is the identity when restricted to the set of pairs of sequences satisfying the Poisson summation formula 
\[
\sum_{n \in \Z} x_{n^2} = \sum_{n \in \Z} y_{n^2}. 
\]
For general sequences, the first entries of this operators possess a correction factor due to the lack of Poisson summation. Indeed, 
it is not difficult to verify that $\text{dim}(\text{ker}(T)) = \text{dim}(\text{coker}(T)) = 1$ from the explicit definitions. Therefore,
we can no longer prove invertibility. 

Nonetheless, we also remark that a direct computation shows that the range of $T$ is \emph{closed.} Therefore, $T$ satisfies all conditions to be a 
Fredholm operator. 

Let us then define a new perturbed operator $S$ defined on $\ell^2_s(\N) \times \ell^2_s(\N),$ such that 
\begin{align*}
S^1(\{x_i\},\{y_i\})_k &= \sum_{n \ge 0} (x_n a_n(\sqrt{k+\eps_k}) + y_n \widehat{a_n}(\sqrt{k+\eps_k})), \cr 
S^2(\{x_i\},\{y_i\})_k &= S^1(\{y_i\},\{x_i\})_k, \cr 
\end{align*} 
for all $k \ge 0,$ where $\eps_k > 0, \, \forall k \ge 0.$ We denote by $\ee_n \in \ell^2_s(\N)$ the vector consisting of $n^{-s}$ on the $n-$th entry, and zero otherwise. With this definition, the set 
\[
\{ (\ee_n,\mathbf{0}) \colon n \in \N\} \cup \{(\mathbf{0},\ee_n) \colon n \in \N\}
\]
forms an orthonormal basis of $\ell^2_s(\N) \times \ell^2_s(\N).$ Thus, 
\[
\|A\|_{HS(\ell^2_s(\N) \times \ell^2_s(\N))}^2 = \sum_{n \in \N} (\|A(\ee_n,\mathbf{0})\|_{(s,s)}^2 + \|A(\mathbf{0},\ee_n)\|_{(s,s)}^2),
\]
where we denote by $\| \cdot \|_{(s,s)}$ the norm of $\ell^2_s(\N) \times \ell^2_s(\N).$ Let then $A = I - \tilde{T}.$ \\

\begin{claim}\label{claim HS} $\|A\|_{HS(\ell^2_s(\N) \times \ell^2_s(\N))} < +\infty$ holds whenever there is $\delta>0$ so that 
$|\eps_k| \lesssim k^{-\frac{5}{4}-\delta}, \, \forall k \ge 1.$ 
\end{claim}

\begin{proof}[Proof of Claim \ref{claim HS}] As mentioned before, we can write the identity on $\ell^2_s(\N) \times \ell^2_s(\N)$ as 
\[
I(\{x_i\},\{y_i\}) = ((x_0,\mathfrak{G}(1),\mathfrak{G}(\sqrt{2}),\dots),(y_0,\widehat{\mathfrak{G}}(1),\widehat{\mathfrak{G}}(\sqrt{2}),\dots)),
\]
where we define the function $\mathfrak{G}$ as in \eqref{eq g function}. With this notation, the operator $\tilde{T}$ becomes 
\[
\tilde{T}(\{x_i\},\{y_i\}) = ((x_0,\mathfrak{G}(\sqrt{1+\eps_1}), \mathfrak{G}(\sqrt{2+\eps_2}), \dots),(y_0, 
\]
\[
\widehat{\mathfrak{G}}(\sqrt{1+\eps_1}),\widehat{\mathfrak{G}}(\sqrt{2+\eps_2}),\dots)).
\]
Therefore, evaluating at the basis vectors gives us 
$$(I-\tilde{T})(\ee_n,\mathbf{0}) = ((0,n^{-s}(a_n(\sqrt{1})-a_n(\sqrt{1+\eps_1}),n^{-s}(a_n(\sqrt{2})-a_n(\sqrt{2+\eps_2})),\dots),
$$
$$(0,0,\dots)).$$ 
We readily see then that 
\begin{align}\label{eq HS explicit} 
\|I-\tilde{T}\|_{HS(\ell^2_s(\N) \times \ell^2_s(\N))}^2 &= \sum_{n > 0} \left(\sum_{k \ge 0} (1+k)^{2s}(1+n)^{-2s}|a_n(\sqrt{k}) - a_n(\sqrt{k+\eps_k})|^2\right) \cr 
							 & + \sum_{n > 0} \left(\sum_{k \ge 0} (1+k)^{2s}(1+n)^{-2s} |\widehat{a_n}(\sqrt{k}) - \widehat{a_n}(\sqrt{k+\eps_k})|^2\right).
\end{align}
From Theorem \ref{eq improvement decay}, we know that 
\begin{align}
|a_n(\sqrt{k}) - a_n(\sqrt{k+\eps_k})|&\le \int_{\sqrt{k}}^{\sqrt{k+\eps_k}} |a_n'(t)| \, \mmd t \cr 
				      &\le \frac{C\eps_k}{\sqrt{k}} n^{3/4} e^{-c\sqrt{k/n}}, \cr
\end{align}
for some $c>0$ and $k \ge 1.$ Analogously, 
\[
|\widehat{a_n}(\sqrt{k}) - \widehat{a_n}(\sqrt{k+\eps_k})| \le \frac{C\eps_k}{\sqrt{k}} n^{3/4} e^{-c\sqrt{k/n}}.
\]

These estimates plus the condition on the $\eps_k$ imply that \eqref{eq HS explicit} may be bounded from above by an absolute constant times
\[
\sum_{n \ge 0} \left(\sum_{k \ge 1} k^{2s} k^{-\frac{5}{2}-2\delta} \cdot k^{-1} e^{-2c\sqrt{k/n}} \right) n^{\frac{3}{2}-2s}.
\]
In order to prove convergence, we first investigate the inner sum. A Riemann sum approach together with a change of variables shows that this is bounded by a constant times
\[
(1+n)^{2s - \frac{5}{2} -2\delta} \left(\int_0^{\infty} t^{2s} t^{-\frac{5}{2}-2\delta} \cdot t^{-1} e^{-c\sqrt{t}} \, \mmd t\right) =: (1+n)^{2s-\frac{5}{2}-2\delta} I_{s,\delta}.
\]
Clearly, the inner integral converges given that $s > \frac{5}{4} + \delta.$ Putting these estimates together with \eqref{eq HS explicit} and using Fubini, we obtain that
\[
\|I-\tilde{T}\|_{HS(\ell^2_s(\N) \times \ell^2_s(\N))}^2 \le I_{s,\delta}\left(\sum_{n \ge 0} (1+n)^{-1-2\delta}\right) < +\infty, 
\]
as desired.
\end{proof}

As a direct corollary, we see that, for each $\delta>0,$ there is $a>0$ so that, if $|\eps_i| \le a i^{-\frac{5}{4}-\delta}\, \forall \, i > 0,$ then 
\[
\|A\|_{HS(\ell^2_s(\N) \times \ell^2_s(\N))} < 1.
\]
In particular, we shall make use of the fact that $T$ is a Fredholm operator by means of such an inequality, with aid of the following result:

\begin{theorem}[Theorems 2.8 and 2.10 in \cite{Schechter}] Let $\Phi(X,Y)$ denote the set of bounded Fredholm operators between Banach spaces $X$ and $Y.$ If $A \in \Phi(X,Y)$ and $K \in \mathcal{K}(X,Y)$ 
is a compact operator, then $A+K \in \Phi(X,Y)$ and $i(A) = i(A+K),$ where we define the \emph{index} $i:\Phi(X,Y) \to \N$ by $i(A) = \text{dim}(\text{ker}(A)) - \text{dim}(\text{coker}(A)) =:\alpha(A) - \beta(A).$ 

Furthermore, if $\|K\|_{op}$ is small enough, then it also holds that $\alpha(A+K) \le \alpha(A).$ 
\end{theorem}

Notice that we may write $S - T = \tilde{T} - I + K_0,$ where $K_0$ has finite rank and bounded, and thus also compact. Therefore, $S = T + (S-T) = T + (\tilde{T} - I) + K_0$ can be written as 
sum of a Fredholm operator $T$ and a compact operator $\tilde{T} - I + K_0.$ This already implies that, modulo a finite-dimensional subspace, the sequences $(\{f(\sqrt{k+\eps_k})\}, \{\widehat{f}(\sqrt{k+\eps_k})\})$ determine 
the sequences $(\{f(\sqrt{k})\},\{\widehat{f}(\sqrt{k})\}).$ That is, we can determine the function $f \in \mathcal{S}_{even}(\R)$ from its (Fourier-)values as 
$\sqrt{k+\eps_k},$ modulo subtracting functions belonging to a finite-dimensional space. 

If, however, we make $|\eps_k| < \epsilon k^{-\frac{5}{4}-\delta},$ with $\epsilon$ small enough, and $|\eps_0| \ll 1,$ we get that the operator norms of 
both $I - \tilde{T} = A$ and $K_0$ can be made arbitrarily small. Thus, 
\[
i(S) = i(T + (S-T)) = i(T)=0 \iff \alpha(S) = \beta(S),
\]
and, moreover, 
$$\alpha(S) \le \alpha(T),$$
as the Hilbert--Schmidt norm of the difference is small. Thus, either 
\[
\alpha(S) = \beta(S) = 0,
\]
in which case we can perfectly invert the operator $S$, or 
\[
\alpha(S) = \beta(S) = 1,
\]
which implies that there is essentially \emph{at most one} function $f_0 \in \mathcal{S}_{even}(\R)$ that vanishes at $\sqrt{k+\eps_k}.$ As 
$(\{f(\sqrt{k+\eps_k})\},\{\widehat{f}(\sqrt{k+\eps_k})\}) \in \text{im}(S)$ for every real $f \in \mathcal{S}_{even}(\R),$ we have proved the followin result.

\begin{theorem}\label{thm alternative} Let $T,S,\{\eps_i\}_{i \ge 0}$ be as above. Then one of the following holds: 
\begin{enumerate}
 \item Either $S$ is an isomorphism from $\ell^2_s(\N) \times \ell^2_s(\N)$ onto itself, and thus the values 
 $$
 (\{f(\sqrt{j+\eps_j})\},\{\widehat{f}(\sqrt{j+\eps_j})\})
 $$
 determine any real function $f \in \mathcal{S}_{even}(\R);$ 
 \item Or $\text{ker}(S)$ has dimension one, and therefore $S$ is an isomorphism from $\text{ker}(S)^{\perp}$ onto $\text{im}(S).$ 
 
 In particular, any real function $f \in \mathcal{S}_{even}(\R)$ is uniquely determined by 
 $$(\{f(\sqrt{j+\eps_j})\},\{\widehat{f}(\sqrt{j+\eps_j})\}),$$ together with the value of 
 \[
 \frac{\langle (\{f(\sqrt{j+\eps_j})\},\{\widehat{f}(\sqrt{j+\eps_j})\}) ,(\{\alpha_i\},\{\beta_i\}) \rangle_{(s,s)}}{\|(\{\alpha_i\},\{\beta_i\})\|_{(s,s)}^2},
 \]
 where $(\{\alpha_i\},\{\beta_i\}) \in \text{ker}(S)$ is a generator for the kernel of $S.$ 
\end{enumerate}
\end{theorem}

Notice that the first option in Theorem \ref{thm alternative} yields immediately an interpolation formula, in the spirit of \eqref{eq inverse}. For the 
second one, the operator is now only invertible if restricted to $\text{ker}(S)^{\perp},$ and now the process of recovering $f \in \mathcal{S}_{even}(\R : \R)$ 
has to take into account the inner product with the kernel vector and the structure of the range.

\subsection{Uniqueness for small powers of integers}\label{sec unique small} Let $\alpha \in (0,1/2).$ We are interested in determining when the only function $f \in \mathcal{S}_{even}(\R)$ that vanishes together with 
its Fourier transform at $\pm n^{\alpha}$ is the identically zero function. \\

Indeed, we would like to study the natural operator that sends the sequence of values at the roots of integers  $(\{f(\sqrt{k})\}_k,\{\widehat{f}(\sqrt{k})\}_k\})$ to the sequence 
$(\{f(n^{\alpha})\}_n,\{\widehat{f}(n^{\alpha})\}_n).$ Our goal is to show that this operator is injective. In order to do that, we will first study simpler operators. 

Indeed, let $K_0 \in \N$ be a fixed positive integer. Fix a set of $2K_0$ positive real numbers $t_1 < t_2 < \dots < t_{2K_0}$ such that $t_1 > \sqrt{K_0}$ and none of the $t_j$ can be written
as a square root of a positive integer. We fix $s>0$ sufficiently large and define the operator 
\begin{align*}
T_{K_0} : \ell^2_s(\N) \times \ell^2_s(\N) \to &\ell^2_s(\N) \times \ell^2_s(\N) \cr 
          (\{x_i\}_i,\{y_i\}_i) \mapsto &((x_0,\mathfrak{G}(t_1),\mathfrak{G}(t_2),\dots,\mathfrak{G}(t_{2K_0}),x_{K_0 +1},x_{K_0 + 2}, \dots),\cr 
          &                               (y_0,\widehat{\mathfrak{G}}(t_1),\widehat{\mathfrak{G}}(t_2),\dots,\widehat{\mathfrak{G}}(t_{2K_0}),y_{K_0+1},y_{K_0+2},\dots)). \cr 
\end{align*}
Here, we denoted by $\mathfrak{G}$ the function defined as in \eqref{eq g function}. 

\begin{lemma}\label{lemma inject k_0} For any $K_0 \ge 1,$ the operator $T_{K_0}$ is bounded and injective. 
\end{lemma}

\begin{proof} We begin with the boundedness assertion. As $T_{K_0}$ differs only in the first $K_0$ coordinates from an interation of the shift operator 
\[
s((\{x_i\}_i,\{y_i\}_i) = ((0,x_0,x_1,\dots),(0,y_0,y_1,\dots)),
\]
boundedness follows from boundedness of the operator that maps a pair of sequences $(\{x_i\}_i,\{y_i\}_i) \in \ell^2_s(\N) \times \ell^2_s(\N)$ into
\begin{align*}
((x_0,\mathfrak{G}(t_1),\mathfrak{G}(t_2),\dots,\mathfrak{G}(t_{2K_0}),0, \dots),\cr 
(y_0,\widehat{\mathfrak{G}}(t_1),\widehat{\mathfrak{G}}(t_2),\dots,\widehat{\mathfrak{G}}(t_{2K_0}),0,\dots)). \cr 
\end{align*}
As $\mathfrak{G}, \widehat{\mathfrak{G}} \in L^{\infty}(\R)$ for any pair of sequences $\{x_i\},\{y_i\},$ with bounds depending only on the $\ell^2_s(\N)-$norms of the sequences, it follows that this new finite-rank 
operator is bounded. 

The injectivity part is subtler. Indeed, fix a pair of sequences $(\{x_i\},\{y_i\}) \in \ell^2_s(\N) \times \ell^2_s(\N)$, and suppose that $T_{K_0}(\{x_i\},\{y_i\}) = 0.$ It follows that the 
special function $\mathfrak{G}(t)$ is a linear combination of $a_1,\dots,a_{K_0},\widehat{a_1},\dots,\widehat{a_{K_0}}.$ In order to analyze such functions, we will need to investigate further the intrinsic 
form of the interpolating functions $a_n,$ and thus those of $b_n^{\pm}.$ 

Indeed, it follows from the Fourier expansion of $g_n^{\pm}$ near infinity and the formula 
\[
b_n^{\pm}(x) = \frac{1}{2} \int_{-1}^1 g_n^{\pm}(z) e^{\pi i x^2 z} \, \mmd z
\]
that, whenever $|x| > \sqrt{n},$ it can also be represented as 
\[
b_n^{\pm}(x) = \sin(\pi x^2) \int_0^{\infty} g_n^{\pm}(1+it) e^{-\pi x^2 t} \, \mmd t. 
\]
As $a_n = (b_n^+ + b_n^-)/2$ and $\widehat{a_n} = (b_n^+ - b_n^-)/2,$ we see that the Fourier invariant part of our function $g$ may be written as 
\[
(\mathfrak{G}+\widehat{\mathfrak{G}})(x) = \sin(\pi x^2) \int_0^{\infty} \left( \sum_{j=1}^{K_0} \alpha_{j} g_j^{+}(1+it)\right) e^{-\pi x^2 t} \, \mmd t,
\]
for some sequence $\alpha_j$ of real numbers, and an analogous identity holds for the $-1$-eigenvalue part $\mathfrak{G}- \widehat{\mathfrak{G}},$ with $g_n^{-}$ instead. We recall that the weakly holomorphic modular forms $g_n^{\pm}$ satisfy that 
\begin{align*}
g_n^+(z) &= \theta(z)^3 P_n^+(1/J(z)), \cr  
g_n^-(z) &= \theta(z)^3 (1-2\lambda(z)) P_n^-(1/J(z)),\cr 
\end{align*}
where the monic polynomials $P_n^-,P_n^+$ are of degree $n$. Therefore, there are polynomials $Q,R$ of degree $\le K_0$ such that 
\begin{align}\label{eq laplace1}
\mathfrak{G} +\widehat{\mathfrak{G}} &= \sin(\pi x^2) \int_0^{\infty} \theta(1+it)^3 Q(1+it) \, e^{-\pi x^2 t} \, \mmd t \cr 
\mathfrak{G} - \widehat{\mathfrak{G}} &= \sin(\pi x^2) \int_0^{\infty} \theta(1+it)^3(1-2\lambda(1+it)) R(1+it) \, e^{-\pi x^2 t} \, \mmd t. \cr 
\end{align}
Before moving forward, we need the following result: 

\begin{lemma}\label{lemma modular} The factors $\theta(1+it)^3$ and $(1-2\lambda(1+it))$ do not change sign for $t \in (0,\infty),$ and the function $1/J(1+it)$ is real-valued and monotonic for $t \in (0,\infty).$ 
\end{lemma}

\begin{proof} By using \eqref{eq theta transform}, we get that 
\[
\theta(1+it) = \sum_{n \in \Z} (-1)^n e^{-\pi n^2 t} = \sum_{n \in \Z} e^{-4\pi n^2 t} - \sum_{n \in \Z} e^{-\pi (2n+1)^2 t}. 
\]
We now consider the function $f_t(x) =  e^{-\pi (2x)^2 t}.$ Then the sum above equals 
\[
\sum_{n \in \Z} f_t(n) - \sum_{n\in \Z} f_t(n+1/2).
\]
By the Poisson summation formula, the difference above equals 
\[
\frac{1}{2\sqrt{t}} \left(\sum_{n \in \Z} e^{-\pi \left(\frac{n}{2\sqrt{t}}\right)^2} - \sum_{n \in \Z} e^{\pi i n}  e^{-\pi \left(\frac{n}{2\sqrt{t}}\right)^2} \right) = \frac{1}{\sqrt{t}} \sum_{n \text{ odd}} e^{-\pi \left(\frac{n}{2\sqrt{t}}\right)^2} \ge 0.
\]
This proves the first assertion. 

For the second, we simply see from \eqref{eq lambda} that $\lambda(1+z)$ has only nonpositive coefficients in its $q-$series expansion. This implies that $\lambda(1+it)$ is nonpositive por $t \in (0,\infty),$ which implies that 
$1-2\lambda(1+it)$ is always nonnegative. 

Finally, for the third assertion, we notice that, as $J(1+z) = \frac{1}{16} \lambda(1+z) (1-\lambda(1+z)),$ and thus, from the analysis above, the $q-$series expansion of $J(1+z)$ contains only nonpositive coefficients. Therefore, the function 
$\frac{1}{J(1+it)}$ is nonpositive for $t \in (0,\infty),$ and it is monotonically decreasing there. This finishes the proof. 
\end{proof}

By Lemma \ref{lemma modular}, we get that the part of the integrand in the expressions above multiplying the $e^{-\pi x^2 t}$ factor changes sign at most $K_0 + 1$ times. Notice that we can embed both integrals in \eqref{eq laplace1}
into the framework of Laplace transforms: denoting 
\[
\mathcal{Q}(t) = \theta(1+it)^3 Q(1+it), \, \mathcal{R}(t) = \theta(1+it)^3 (1-2\lambda(1+it))R(1+it),
\]
we are interested in studying the positive zeros of $\mathcal{L}[\mathcal{Q}](\pi x^2), \mathcal{L}[\mathcal{R}](\pi x^2),$ where 
\[
\mathcal{L}[\phi](s) = \int_0^{\infty} \phi(t) e^{-st} \, \mmd t 
\]
denotes the Laplace transform of $\phi$ evaluated at the point $s.$ We may reduce even further our task to studying the positive zeros of 
$\mathcal{L}[\mathcal{Q}], \mathcal{L}[\mathcal{R}].$ The following result, a version of the Descartes rule for the Laplace transform, is the tool we need to bound the number of positive zeros 
of such expressions as a function of their number of changes of signs. 

\begin{proposition}[Descartes rule for the Laplace transform]\label{prop rule of signs} Let $\phi:\R \to \R$ be a smooth function such that its Laplace transform $\mathcal{L}[\phi]$ converges on some open half-plane 
$\text{Re}(s) > s_0.$ Then the number of zeros of $\mathcal{L}[\phi]$ on the interval $(s_0,+\infty)$ is at most the number of sign changes of $\phi.$ 
\end{proposition}

\begin{proof} The proof follows by induction on the number of sign changes of the function $\phi.$ Indeed, if $\phi \ge0,$ it follows easily that the Laplace transform $\mathcal{L}[\phi] \ge 0,$ with equality if and only if $\phi \equiv 0.$ 

Suppose now that $\phi$ changes sign $n+1$ times on $(0,\infty).$  Number its zeros on the positive half-line as $s_0<s_1<\cdots<s_n.$ 
Then $\mathcal{L}[\phi]$ has as many zeros as $e^{s_0 t} \mathcal{L}[\phi](t) = F(t).$ The derivative of $F$ is then given by 
\[
F'(t) = -\int_0^{\infty} (s-s_0)\phi(s) e^{-(s-s_0)t} \, \mmd s = e^{s_0 t} \mathcal{L}[(s-s_0)\phi(s)](t). 
\]
Notice that the new smooth function $(s-s_0) \phi(s)$ still satisfies the same properties as $\phi$, but now has exactly $n$ sign changes. By inductive hypothesis, 
$F'$ has at most $n$ zeros, which, by the mean value theorem, implies that $F$ has at most $n+1$ zeros. This finishes the proof. 
\end{proof}

Using this claim for $\mathcal{Q},\mathcal{R},$ we see that their respective Laplace transform possess at most $K_0$ zeros on the interval $(\sqrt{K_0},+\infty).$ With this information, we can already finish: from \eqref{eq laplace1}, the 
functions $\mathfrak{G}\pm \widehat{\mathfrak{G}}$ can only vanish at at most $K_0$ points on the interval $(\sqrt{K_0},\infty)$ which are not roots of positive integers, in case $\mathfrak{G} \not\equiv 0.$ But, according to our asumption that 
$(\{x_i\},\{y_i\}) \in \text{ker}(T_{K_0}),$ we have $\mathfrak{G}(t_j) = \widehat{\mathfrak{G}}(t_j) = 0, j=1,\dots,2K_0.$ By the properties we chose for the sequence $t_j,$ $\mathfrak{G} \equiv 0,$ and thus the map $T_{K_0}$ is injective.
\end{proof}

We need one more result in order to use our methods to infer results about uniqueness for small powers of integers. In contrast to the full perturbation case of our main theorem, we must prove that the injective operators $T_{K_0}$ are also somewhat stable with respect to
injectivity under perturbations. In order to do this, the following result is essencial. 

\begin{lemma}\label{lemma closed range} The range of $T_{K_0}$ is closed. 
\end{lemma}

\begin{proof} Suppose the sequence in $\ell^2_s(\N) \times \ell^2_s(\N)$ given by $\{T_{K_0}(\{x_i^j\},\{y_i^j\})\}_{j \ge 0}$ is a Cauchy sequence. This implies that the sequence
$\{\{x_i^j\}_{i=0,K_0+1,\dots},\{y_i^j\}_{i=0,K_0+1,\dots}\}_{j \ge 0}$ is a Cauchy sequence, and therefore it converges to a certain limiting sequence 
$$\{\{x_i\}_{i=0,K_0+1,\dots},\{y_i\}_{i=0,K_0+1,\dots}\} \in \ell^2_s(\N) \times \ell^2_s(\N).$$
Define, thus, the $4K_0 \times 2K_0$ matrix $A_{K_0}$ given by taking 
\[
(a_1(t_j),a_2(t_j),\dots,a_{K_0}(t_j),\widehat{a_1}(t_j),\widehat{a_2}(t_j),\dots,\widehat{a_{K_0}}(t_j))
\]
and
\[
(\widehat{a_1}(t_j),\widehat{a_2}(t_j),\dots,\widehat{a_{K_0}}(t_j),a_1(t_j),a_2(t_j),\dots,a_{K_0}(t_j))
\]
to be its lines, for $j=1,\dots,2K_0.$ We first claim that this matrix is injective. Indeed, 
\[
\tilde{\mathfrak{G}}(t) = \sum_{i=1}^{K_0} (x_i a_i(t) + y_i \widehat{a_i}(t))
\]
vanishes, together with its Fourier transform, at $t_j$,$j=1,\dots,2K_0,$ where $(\{x_i\}_{i=1}^{K_0},\{y_i\}_{i=1}^{K_0})$ belongs to $\text{ker}(A_{K_0}).$ By the proof of Lemma \ref{lemma inject k_0}, this implies 
$x_i = y_i = 0, i=1,\cdots,K_0.$ 

As $A_{K_0}$is injective, there is a constant $c_{K_0} >0$ so that 
\begin{equation}\label{eq inject matrix}
 \|A_{K_0} \mathbf{v}\|_{4K_0} \ge c_{K_0} \|\mathbf{v}\|_{2K_0},
\end{equation}
where we denote by $\| \cdot \|_{d}$ the usual euclidean norm on a $d-$dimensional space. Translating to our original problem, as  $\{T_{K_0}(\{x_i^j\},\{y_i^j\})\}_{j \ge 0}$ 
is a Cauchy sequence in $\ell^2_s(\N) \times \ell^2_s(\N),$   
$$\{\{x_i^j\}_{i=0,K_0+1,\dots},\{y_i^j\}_{i=0,K_0+1,\dots}\}_{j \ge 0}$$
is a convergent sequence, and thus we get that the sequences 
\[
\sum_{i=1}^{K_0} (x_i^k a_i(t_j) + y_i^k \widehat{a_i}(t_j)), \, j=1,\dots,2K_0
\]
are also Cauchy in $k \ge 0$. By \eqref{eq inject matrix}, $(\{x_i^k\}_{i=1}^{K_0},\{y_i\}_{i=1}^{K_0})_{k \ge 0}$ is Cauchy. 
This implies that there is a limiting sequence $(\{x_i\},\{y_i\}) \in \ell^2_s(\N) \times \ell^2_s(\N)$ so that 
\[
T_{K_0}(\{x_i^j\},\{y_i^j\}) \to T_{K_0}(\{x_i\},\{y_i\}), \, \text{ as } j \to \infty.
\]
This finishes the proof. 
\end{proof}

We are finally able to prove the following uniqueness result: 

\begin{corollary}\label{corol diagonal} Let $\alpha \in (0,\frac{2}{9}).$ There exists $c_{\alpha} > 0$ so that $\forall c < c_{\alpha},$ if $f \in \mathcal{S}_{even}(\R)$ is a real function that vanishes together with its Fourier transform 
at $\pm c_{\alpha} n^{\alpha},$ then $f \equiv 0.$ 
\end{corollary}

\begin{proof} We start by noticing that, whenever $n \in \N$ is sufficiently large, then there is $m \in \N$ so that $|\sqrt{n} - c m^{\alpha}| \lesssim c^{\frac{1}{\alpha}} n^{\frac{\alpha-1}{2\alpha}}.$ Indeed, 
we simply let $m= \lfloor (n/c^2)^{\frac{1}{2\alpha}} \rfloor.$ We get that 
\begin{align*}
|\sqrt{n} - c m^{\alpha}| = c \alpha \int_{\lfloor (n/c^2)^{1/(2\alpha)} \rfloor}^{(n/c^2)^{1/(2\alpha)}} t^{\alpha - 1} \, \mmd t \lesssim c^{1/\alpha} \alpha n^{\frac{\alpha-1}{2\alpha}}. 
\end{align*}
In particular, if $\frac{\alpha - 1}{2\alpha} < - \frac{5}{4} - \frac{1}{2} \iff \alpha < \frac{2}{9},$ then for all $n \ge n_0(\alpha),$ there exists $m \in \N$ so that we can write 
$m^{\alpha} = \sqrt{n + \eps_n},$ where $\eps_n$ satisfies the conditions of Theorem \ref{thm mainthm}. Let us single out the sequence of numbers selected above, which we index as $\{m(n)^{\alpha}\}_{n \ge n_0(\alpha)}.$ 
We then consider the operator $T_{n_0(\alpha)}$ associated to some sequence of $2n_0(\alpha)$ 
positive real numbers $t_j, j =1,\dots,2n_0(\alpha),$ satisfying the hypotheses of Lemma \ref{lemma inject k_0}. 

We claim that the \emph{perturbed} operator 
\begin{align}\label{eq perturbed new}
\tilde{T}_{n_0(\alpha)} &: \ell^2_s(\N) \times \ell^2_s(\N) \to \ell^2_s(\N) \times \ell^2_s(\N)  \cr 
			& (\{x_i\},\{y_i\}) \mapsto ((x_0,\mathfrak{G}(t_1),\mathfrak{G}(t_2),\dots,\mathfrak{G}(t_{2n_0}),\mathfrak{G}(m(n_0+1)^{\alpha}),\mathfrak{G}(m(n_0+2)^{\alpha}),\dots),\cr
			&                            (y_0,\widehat{\mathfrak{G}}(t_1),\widehat{\mathfrak{G}}(t_2),\dots,\widehat{\mathfrak{G}}(t_{2n_0}),\widehat{\mathfrak{G}}(m(n_0+1)^{\alpha}),\widehat{\mathfrak{G}}(m(n_0+2)^{\alpha}),\dots))\cr
\end{align}
is injective. Indeed, from Lemma \ref{lemma closed range} there must exist a constant $C_{n_0}$ so that 
\[
\|T_{n_0} \mathbf{v}\|_{(s,s)} \ge C_{n_0} \|\mathbf{v}\|_{(s,s)}                   
\]
holds for all $\mathbf{v} \in \ell^2_s(\N) \times \ell^2_s(\N).$ But, by the same calculation as in the previous subsection, we have that 
\[
\|\tilde{T}_{n_0(\alpha)} - T_{n_0(\alpha)}\|_{HS(\ell^2_s(\N) \times \ell^2_s(\N))} < C_{n_0}/2
\]
holds, as long as we take $\alpha < \frac{2}{9}$ and $c=c(\alpha)$ sufficiently small. This implies, in particular, that 
\[
\|\tilde{T}_{n_0} \mathbf{v}\||_{(s,s)} \ge \frac{C_{n_0}}{2} \|\mathbf{v}\|_{(s,s)},
\]
for each $\mathbf{v} \in \ell^2_s(\N) \times \ell^2_s(\N),$ and thus the operator $\tilde{T}_{n_0}$ is, indeed, injective, as desired. 

In order to conclude, we notice that the operator 
\begin{align}\label{eq perturbed new2}
& \mathcal{T}_{n_0(\alpha)} : \ell^2_s(\N) \times \ell^2_s(\N) \to \ell^2_s(\N) \times \ell^2_s(\N)  \cr 
			& (\{x_i\},\{y_i\}) \mapsto ((x_0,\mathfrak{G}(ck_1^{\alpha}),\mathfrak{G}(ck_2^{\alpha}),\dots,\mathfrak{G}(ck_{2n_0}^{\alpha}),\mathfrak{G}(m(n_0+1)^{\alpha}),\mathfrak{G}(m(n_0+2)^{\alpha}),\dots),\cr
			&                            (y_0,\widehat{\mathfrak{G}}(ck_1^{\alpha}),\widehat{\mathfrak{G}}(ck_2^{\alpha}),\dots,\widehat{\mathfrak{G}}(ck_{2n_0}^{\alpha}),\widehat{\mathfrak{G}}(m(n_0+1)^{\alpha}),\widehat{\mathfrak{G}}(m(n_0+2)^{\alpha}),\dots))\cr
\end{align}
for some sequence $k_j, j=1,\dots,2n_0$ of integers not belonging to the sequence $m(n)$ we selected above, is still injective. In fact, it only differs from the operator $\tilde{T}_{n_0}$ in at most the first
$2n_0$ entries. But, on the other hand, for $k_j = \lfloor (t_j/c)^{1/\alpha} \rfloor, j =1,\dots,2n_0,$ and $c>0$ sufficiently small, we see that 
\begin{align*}
|\mathfrak{G}(ck_j^{\alpha}) - \mathfrak{G}(t_j)| &\le \sum_{i=0}^{\infty} (x_i|a_i(t_j) - a_i(ck_j^{\alpha})| + y_i|\widehat{a_i}(t_j) - \widehat{a_i}(ck_j^{\alpha})|) \cr 
                            &\lesssim \sup_{j \in [0,2n_0]} |t_j - ck_j^{\alpha}| \left(\sum_{i=0}^{\infty} i^{5/2} (|x_i| + |y_i|) \right) \cr 
                            &\lesssim \epsilon \| (\{x_i\},\{y_i\})\|_{(s,s)}.
\end{align*}
For $\epsilon > 0$ sufficiently small, we see from the previous argument that $\mathcal{T}_{n_0(\alpha)}$ still has closed range and is injective. This readily implies that the sequence 
$(\{f(\pm n^{\alpha})\},\{\widehat{f}(\pm n^{\alpha})\})$ determines uniquely the sequence $(\{f(\sqrt{n})\},\{\widehat{f}(\sqrt{n})\}).$ This finishes the proof. 
\end{proof}

One can inquire about the importance of such a result, as in \cite{RS1} we have shown that the uniqueness result stated in Corollary \ref{corol diagonal} 
hold for $\alpha \in (0,1-\sqrt{2}/2),$ which is significantly larger than the range stated here. Nonetheless, Corollary \ref{corol diagonal} 
gives us \emph{automatic} results. Indeed, if one manages to prove that for all $\delta >0$ there is $\epsilon > 0$ so that, if $|\eps_k| \le \epsilon, \, 
\forall k \in \N,$ then 
\[
\|I-\tilde{T}\|_{op} < \delta,
\]
it implies automatically that we can extend the results in Corollary \ref{corol diagonal} to the full diagonal range $\alpha \in (0,1/2).$ 

We also note that Corollary \ref{corol diagonal} is not all we can say about the problem of determining the best exponents $(\alpha,\beta)$ so that 
\[
f(\pm n^{\alpha}) = \widehat{f}(\pm n^{\beta}) = 0, \, f \in \mathcal{S}_{even}(\R) \Rightarrow f \equiv 0. 
\]
Indeed, we can easily go further than the diagonal case exposed above: if $\alpha, \beta \in (0,2/9)$ are an arbitrary pair of exponents, we notice that we can still
pick $n_0 \in \N$ so that for each $n > n_0 = n_0(\alpha,\beta),$ there exists a pair $(m_1(n),m_2(n)) \in \N^2$ so that 
\[
|cm_1(n)^{\alpha} - \sqrt{n}| + |cm_2(n)^{\beta} - \sqrt{n}| \lesssim c^{1/\alpha} \alpha n^{\frac{\alpha - 1}{2\alpha}} + c^{1/\beta} \beta n^{\frac{\beta-1}{2\beta}}.
\]
This induces us to consider the operator 
\begin{align}\label{eq perturbed newtwo}
&\mathcal{T}_{n_0(\alpha,\beta)} : \ell^2_s(\N) \times \ell^2_s(\N) \to \ell^2_s(\N) \times \ell^2_s(\N)  \cr 
			& (\{x_i\},\{y_i\}) \mapsto ((x_0,\mathfrak{G}(ck_1^{\alpha}),\mathfrak{G}(ck_2^{\alpha}),\dots,\mathfrak{G}(ck_{2n_0}^{\alpha}),\mathfrak{G}(m_1(n_0+1)^{\alpha}),\mathfrak{G}(m_1(n_0+2)^{\alpha}),\dots),\cr
			&                            (y_0,\widehat{\mathfrak{G}}(cl_1^{\beta}),\widehat{\mathfrak{G}}(cl_2^{\beta}),\dots,\widehat{\mathfrak{G}}(cl_{2n_0}^{\beta}),\widehat{\mathfrak{G}}(m_2(n_0+1)^{\beta}),\widehat{\mathfrak{G}}(m_2(n_0+2)^{\beta}),\dots))\cr
\end{align}
for two sequences of integers $(k_j,l_j), j=1,\dots,2n_0,$ so that $|t_j - ck_j^{\alpha}| + |t_j - cl_j^{\beta}|$ is sufficiently small for all $j \in [0,2n_0],$ 
where we select $t_j, j=1,\dots,2n_0$ satisfying the hypotheses of Lemma \ref{lemma inject k_0}.

By the same strategy outlined in the proof of Corollary \ref{corol diagonal}, the Hilbert-Schmidt norm as operators acting on $\ell^2_s(\N) \times \ell^2_s(\N)$ of the difference 
$T_{n_0(\alpha,\beta)} - \mathcal{T}_{n_0(\alpha,\beta)}$ is arbitrarily small, as long as we make the value of $c=c(\alpha,\beta)$ smaller. As a consequence, $\mathcal{T}_{n_0}$ is 
also injective and its range is closed. These considerations prove, therefore, the following:

\begin{corollary} Let $\alpha,\beta \in (0,2/9).$ Then there is $c_{\alpha,\beta} > 0$ so that for all $c < c_{\alpha,\beta},$ if $f \in \mathcal{S}_{even}(\R)$ is a real function that vanishes 
at $\pm c n^{\alpha}$ and its Fourier transform vanishes at $\pm c n^{\beta},$ then $f \equiv 0.$ 
\end{corollary}

\noindent\textbf{Remark.} In the end, we do not quite attain the primary goal of this section of proving Fourier uniqueness results for the sequences $(\{\pm n^{\alpha}\},\{\pm n^{\beta}\}),$ but only a slightly 
weaker version of it, with a small constant $c(\alpha,\beta)$ in front. The main reason for that in the proofs above is the location of the positive reals $t_i:$ although their 
exact values do not matter in the end, it is crucial, in order to use Proposition \ref{prop rule of signs}, that they lie \emph{after} the node $n_0.$ We must therefore either force $n_0$ not to be
too large in order not to make the norm of the matrix $A_{K_0}$ too small, or fix them from the beginning and make the perturbations of $T_{K_0}$ fall closer to it. In any case, this implies 
nontrivial use of the constant $c$ multiplying the sequences $(\{\pm n^{\alpha}\},\{\pm n^{\beta}\}).$ 

We believe that further studying operators resembling $T_{K_0}$ above and their injectivity properties could yield better results in this regard. In order not to make this exposition 
even longer, we will not pursue this matter any further.

\subsection{The Cohn-Kumar-Miller-Radchenko-Viazovska result and perturbed interpolation formulae with derivatives}\label{sec ckmrv}

As another illustration of our main technique, we prove that the interpolation formulae with derivatives in dimension $8$ and $24$ from \cite{CKMRV2} can be suitably 
perturbed. 

Indeed, we first recall one of the main results of \cite{CKMRV2}: let $(d,n_0)$ be either $(8,1)$ or $(24,2).$ Then every $f \in \mathcal{S}_{rad}(\R^d)$ can be uniquely recovered by the sets of values 
\[
\{ f(\sqrt{2n}), f'(\sqrt{2n}), \widehat{f}(\sqrt{2n}), \widehat{f}'(\sqrt{2n})\}, \, n \ge n_0,
\]
through the interpolation formula
\begin{align}\label{eq interpol derivative} 
f(x) & = \sum_{n \ge n_0} f(\sqrt{2n}) a_n(x) + \sum_{n \ge n_0} f'(\sqrt{2n}) b_n(x) \cr
     & + \sum_{n \ge n_0} \widehat{f}(\sqrt{2n}) \widehat{a_n}(x) +  \sum_{n \ge n_0} \widehat{f}'(\sqrt{2n}) \widehat{b_n}(x). \cr 
\end{align}
We also have uniform estimates on the functions $a_n,\widehat{a_n},b_n,\widehat{b_n}:$ indeed, there is $\tau > 0$ so that 
\begin{equation}\label{eq uniform derivatives}
\sup_{l \in \{0,1,2\}} \sup_{x \in \R^d} (1+|x|)^{100} \left(|a_n^{(l)}(x)| + |\widehat{a_n}^{(l)}(x)| + |b_n^{(l)}(x)| + |\widehat{b_n}^{(l)}(x)|\right) \lesssim n^{\tau},
\end{equation}
for all $n \in \N.$ Here and throughout this section, we shall denote by $g'(x)$ the derivative of the (radial) function $g$ regarded as a one-dimensional function. 

By \cite[Theorem~1.9]{CKMRV2}, we know that the matrices
\begin{equation}\label{eq matrix derivate}
M_n(x) = \begin{pmatrix}
       a_n(x) & a_n'(x) & \widehat{a_n}(x) & \widehat{a_n}'(x) \\ 
       b_n(x) & b_n'(x) & \widehat{b_n}(x) & \widehat{b_n}'(x) \\
       \widehat{a_n}(x) & \widehat{a_n}'(x) & a_n(x) & a_n'(x) \\
       \widehat{b_n}(x) & \widehat{b_n}'(x) & b_n(x) & b_n'(x) \\
       \end{pmatrix}
\end{equation}
satisfy that $M_n(\sqrt{2m}) = \delta_{mn} I_{4 \times 4}.$ As we know that the map that takes a vector of sufficiently rapidly decaying sequences 
$$(\{\alpha_n\},\{\beta_n\},\{\tilde{\alpha}_n\},\{\tilde{\beta}_n\})$$
onto the function 
\[
\mathfrak{f}(x) = \sum_{n \ge n_0} \left( \alpha_n a_n(x) + \beta_n b_n(x) + \tilde{\alpha}_n \widehat{a_n}(x) + \tilde{\beta}_n \widehat{b_n}(x) \right)
\]
is, in fact, injective (and moreover an isomorphism if we consider the set of all arbitrarily rapidly decaing sequences), we shall make use of this function in our estimates. Indeed, 
we have that the map that takes the quadruple of sequences $$(\{\alpha_n\},\{\beta_n\},\{\tilde{\alpha}_n\},\{\tilde{\beta}_n\})$$ onto 
\[
(\ff(\sqrt{2n}),\ff'(\sqrt{2n}),\widehat{\ff}(\sqrt{2n}), \widehat{\ff}'(\sqrt{2n}))_{n \ge n_0} 
\]
is, in fact, the identity. Another way to represent this map is as the series 
\[
\sum_{n \ge n_0} (\alpha_n,\beta_n,\tilde{\alpha}_n,\tilde{\beta}_n) \cdot M_n(\sqrt{2n}).
\]
We define, therefore, the operator that takes the same quadruple onto 
\[
(\ff(\sqrt{2n+\eps_n}),\ff'(\sqrt{2n+\eps_n}),\widehat{\ff}(\sqrt{2n+\eps_n}), \widehat{\ff}'(\sqrt{2n+\eps_n}))_{n \ge n_0}.
\]
In the alternative notation, this operator, which we shall denote by $\mathfrak{T},$ is given by 
\[
\sum_{n \ge n_0} (\alpha_n,\beta_n,\tilde{\alpha}_n,\tilde{\beta}_n) \cdot M_n(\sqrt{2n+\eps_n}).
\]

As before, we seek to prove that $\mathfrak{T}$ is invertible when defined over some space 
\[
\ell^2_s(\N) \times \ell^2_s(\N) \times \ell^2_s(\N) \times \ell^2_s(\N) =: (\ell^2_s(\N))^4,
\]
where we may take $s \gg 1$ sufficiently large. As our aim here is not to establish the sharpest possible results, but only to prove that we may perturb the aforementioned interpolation formulae, we 
shall make use of the Hilbert--Schmidt test, as in \S \ref{sec origin} above. Indeed, we wish to prove that 
\[
\| I - \mathfrak{T} \|_{HS((\ell^2_s(\N))^4)} < 1.
\]
A simple computation with the Hilbert--Schmidt norm using \eqref{eq matrix derivate} shows that this quantity is bounded by 
\begin{align*}
\sum_{m,n >0} & m^{2s} n^{-2s} ((|a_n(\sqrt{2m}) - a_n(\sqrt{2m+\eps_m})|^2 + |\widehat{a_n}(\sqrt{2m})- \widehat{a_n}(\sqrt{2m+\eps_m})|^2 + \cr 
	      & + |a_n'(\sqrt{2m}) - a_n'(\sqrt{2m+\eps_m})|^2 + |\widehat{a_n}'(\sqrt{2m}) - \widehat{a_n}'(\sqrt{2m+\eps_m})|^2 + \cr
	      & |b_n(\sqrt{2m}) - b_n(\sqrt{2m+\eps_m})|^2 + |\widehat{b_n}(\sqrt{2m}) - \widehat{b_n}(\sqrt{2m + \eps_m})|^2 + \cr 
	      & + |b_n'(\sqrt{2m}) - b_n'(\sqrt{2m+\eps_m})|^2 + |\widehat{b_n}'(\sqrt{2m}) - \widehat{b_n}'(\sqrt{2m+\eps_m})|^2).\cr
\end{align*}
By \eqref{eq uniform derivatives} and the mean value theorem, the sum above is bounded by (an absolute constant times) 
\[
\sum_{m,n > 0} m^{2s} n^{-2s} \times m^{-100} n^{2\tau} \eps_m^2.
\]
The sum above is representable as a product of a sum in $m$ and one in $n.$ The one in $n$ is convergent if $s > \tau + 1.$ We then fix such a value of $s.$ For such values, the second sum is 
\[
\sum_{m > 0} m^{2s-100} \eps_m^2,
\]
which converges in case $\eps_m \lesssim m^{49-s}.$ For all such sequences, the difference $I - \mathfrak{T}$ is a Hilbert--Schmidt operator. Moreover, if $\eps_m \le \delta m^{49-s}$ for $\delta > 0$ 
sufficiently small, we will have $\|I-\mathfrak{T}\|_{HS(\ell^2_s(\N)^4)} < 1.$ Summarizing, we have shown the following result: 

\begin{theorem}\label{thm perturb derivatives} There is $C_0 > 0$ so that the following holds: there is $\delta >0$ so that, for each sequence $\eps_k$ so that $|\eps_k| < \delta k^{-C_0},$ then any function 
$f \in \mathcal{S}_{rad}(\R^d)$ is uniquely determined by the values 
\begin{equation}\label{eq perturbed derivative wow}
\left(f(\sqrt{2n+\eps_n}),f'(\sqrt{2n+\eps_n}), \widehat{f}(\sqrt{2n+\eps_n}), \widehat{f}'(\sqrt{2n+\eps_n})\right)_{n \ge n_0},
\end{equation}
where we let $(d,n_0) = (8,1) \text{ or } (24,2).$ 
\end{theorem} 

In the same spirit of \S \ref{sec main result}, one can obtain an interpolation formula with the values \eqref{eq perturbed derivative wow} from Theorem \ref{thm perturb derivatives}. 

We remark that, in the same way that we undertook our analysis for the Radchenko-Viazovska interpolating functions, we expect the functions $a_n,b_n$ in \cite[Theorem~1.9]{CKMRV2} should also satisfy 
some exponential-like decay. This fact, although possible, should be sensibly more technically involved than Theorem \ref{eq improvement decay}, due to the more complicated nature of the construction of the
interpolating functions with derivatives in dimensions 8 and 24. 

\subsection{Perturbed interpolation formulae for odd functions} Finally, in the same spirit of the results in Section \ref{perturbed radchenko viazovska}, we briefly comment on interpolation formulae for 
odd functions. Recall the following results from \cite[Section~7]{RV}: 
\begin{theorem}[Theorem 7 in \cite{RV}] There exist sequences of odd functions $d_m^{\pm} : \R \to \R, \, m \ge 0,$ belonging to the Schwartz class so that 
\[
\widehat{d_m^{\pm}} = (\mp i) d_m^{\pm}, \, d_m^{\pm}(\sqrt{n}) = \delta_{n,m} \sqrt{n}, \, n \ge 1.
\]
Moreover, $\lim_{x \to 0} \frac{d_m^+(x)}{x} = \delta_{0m}.$ These functions satisfy the uniform bound 
\[
|d_n^{\pm}(x)| \lesssim n^{5/2}, \, \forall x \in \R, n \ge 0,
\]
and, finally, for each odd and real Schwartz function $f:\R \to \R,$ 
\begin{equation}\label{eq interpolation odd} 
f(x) = d_0^+(x) \frac{f'(0) + i \widehat{f}'(0)}{2} + \sum_{n \ge 1} \left( c_n(x) \frac{f(\sqrt{n})}{\sqrt{n}} - \widehat{c_n}(x) \frac{\widehat{f}(\sqrt{n})}{\sqrt{n}} \right),
\end{equation}
where $c_n = (d_n^+ + d_n^-)/2,$ and the right-hand side of the sum above converges absolutely. 
\end{theorem}
As a direct consequence, we see that any real, odd, Schwartz function on the real line is determined uniquely by the union of its values at $\sqrt{n}$ and the values of its Fourier transform 
at $\sqrt{n}$ with $f'(0)$ and $\widehat{f}'(0).$ By employing the results in Section \ref{perturbed radchenko viazovska}, 
we will show that we can actually recover any such function from $\{f(\sqrt{n + \eps_n})\}_{n \ge 1} \cup \{\widehat{f}(\sqrt{n+\eps_n})\}_{n \ge 1} \cup \{f'(0)\}\cup\{\widehat{f}'(0)\}$ instead. 

Indeed, first of all, we start by noticing that the same techniques employed to refine the uniform estimates from Radchenko--Viazovska \cite{RV} can be applied to the functions $d_m^{\pm},$ as they are defined in a completely analogous way to the 
$b_n^{\pm}$ from Section \ref{perturbed radchenko viazovska}. By carrying out the same kind of estimates, we are able to obtain
\begin{equation}\label{eq improvement odd}
|d_n^{\pm}(x)| \lesssim n^{3/4} \log^{3/2} (1+n) e^{-c'|x|/\sqrt{n}}, \, \forall x \in \R, \, n \ge 1,
\end{equation}
for some absolute constant $c'>0.$ By the same analysis of the $\partial_x-$partial derivative of the generating function used in \S \ref{sec improve}, this readily implies that the derivatives of the $d_n^{\pm}$ satisfy morally the same decay; in fact, 
$|(d_n^{\pm})'(x)| \lesssim n^{5/4} \log^{3/2}(1+n) e^{-c''|x|/\sqrt{n}}, \, \forall x \in \R, \, n \ge 1,$ with $c''>0$ another absolute constant. 

We consider now the operator that takes a pair of sequences $(\{\alpha_n\},\{\beta_n\}) \in \ell^2_s(\N) \times \ell^2_s(\N)$, $s > 0$ to be chosen, into 
\[
\left\{\sum_{n \ge 0} (\alpha_n,\beta_n) C_n(\sqrt{m+\eps_m})\right\}_{m \ge 0}, 
\]
where we abbreviate $C_n(x) = \begin{pmatrix} 
                             \frac{c_n(x)}{\sqrt{n}} & \frac{\widehat{c_n}(x)}{\sqrt{n}} \\ 
                             - \frac{\widehat{c_n}(x)}{\sqrt{n}} & \frac{c_n(x)}{\sqrt{n}} \\ 
                             \end{pmatrix} $. Let us denote this operator by $\mathcal{V}.$ From \eqref{eq interpolation odd} and the fact that the function 
$d_0^+(x) = \frac{\sin(\pi x^2)}{\sinh(\pi x)}$ \emph{vanishes} together with its Fourier transform at $\pm \sqrt{n}, \, n \in \N,$ we know that the identity operator on $\ell^2_s(\N) \times \ell^2_s(\N)$ 
may be written as 
\[
\left\{\sum_{n \ge 0} (\alpha_n,\beta_n) C_n(\sqrt{m})\right\}_{m \ge 0}.
\]
Therefore, the techniques from \S \ref{sec main result}, \S \ref{sec ckmrv} and \ref{sec origin}, together with our previous considerations in this subsection,
allow us to deduce the following result: 
\begin{theorem}\label{thm perturb odd} There is $\delta > 0$ so that, in case $|\eps_n| \le \delta n^{-\frac{7}{4}},$ then for each $f \in \mathcal{S}_{odd}(\R)$ real, the values 
\[
\left(f(\sqrt{1+\eps_n}), f(\sqrt{2+\eps_2}),\dots\right)
\]
and 
\[
\left(\widehat{f}(\sqrt{1+\eps_n}), \widehat{f}(\sqrt{2+\eps_2}),\dots\right)
\]
allow us to recover uniquely the values $\left(f(1),f(\sqrt{2}),f(\sqrt{3}),\dots\right)$ and 
\linebreak
$\left(\widehat{f}(1),\widehat{f}(\sqrt{2}),\widehat{f}(\sqrt{3}),\dots\right).$ In particular, given the values 
$$\{f(\sqrt{n + \eps_n})\}_{n \ge 1} \cup \{\widehat{f}(\sqrt{n+\eps_n})\}_{n \ge 1}\cup \{f'(0)\}\cup\{\widehat{f}'(0)\},$$
we can uniquely recover any real function $f \in \mathcal{S}_{odd}(\R).$ 
\end{theorem}

As previously mentioned, we do not carry out the details here, for their similarities with the proof of theorems \ref{eq improvement decay} and \ref{thm mainthm}.

\section{Comments and Remarks}\label{sec final}

In this section, we gather some remarks about the problems and techniques discussed, as well as state some results we expect to be true.

\subsection{Maximal perturbed Interpolation Formulae for Band-limited functions} 
In Section \ref{sec BL}, we have seen how our basic functional analysis techniques can be employed in order to deduce new interpolation formulae for band-limited functions. Although Kadec's proof also uses the basic fact that, whenever a perturbation of the identity is sufficiently small, then we can basically `invert' an operator, he then proceeds to 
find that the set of exponentials $\{\exp(2\pi i (n+\eps_n) x)\}_{n \ge 0}$ is a Riesz basis for $L^2(-1/2,1/2)$ if $\sup_n |\eps_n| < 1/4$ by means of \emph{orthogonality} considerations. Indeed, 
one key strategy in his estimates is to expand in the different complete orthogonal system 
\[
\{1,\cos(2\pi n t), \sin((2n-1)\pi t)\}_{n \ge 1}
\]
and use the properties of this expansion. Our results, as much as they do not come so close to Kadec's threshold, follow a slightly different path: instead of using the orthogonality of a different system, 
we choose to work directly with discrete analogues of the Hilbert transform and estimate over those. Although we do not reach -- by a $0.011$ margin -- the sharp $1/4-$perturbation result, 
one advantage of our approach is that it yields bounds for perturbing \emph{any} kind of interpolation formulae with derivatives. Indeed, following the line of thought of Vaaler, many authors have investigated 
the property of recovering the values of a function $f \in L^2(\R)$ band-limited to $[-k/2,k/2]$ from the values of its $(k-1)-$first derivatives (see, e.g., \cite{Lit02} and \cite{FelipeLittman}). Our approach
in \S \ref{sec BL} in order to prove Theorem \ref{thm vaaler perturb} generalizes easily to the case of several derivatives by an easy modification. It can be summarized as follows: 
\begin{theorem}\label{thm derivatives bl} There is $L(k) > 0$ so that if $\sup_{n \in \Z} |\eps_n| < L(k),$ then any function $f \in L^2(\R)$ band-limited to $[-k/2,k/2]$ is uniquely determined by the values of 
\[
f^{(l)}(n+\eps_n), \, n \in \Z, l=0,1,\dots,k-1.
\]
\end{theorem}
A natural question that connects our results to Kadec's results is about the \emph{best} value of $L(k)$ so that Theorem \ref{thm derivatives bl} holds. We do not have evidence to back any concrete
conjecture, but we find possible that the threshold $L(k) = \frac{1}{4}$ is kept for higher values of $k \in \N.$ We speculat that, in order to prove such a result, one would need to find an appropriate 
hybrid of our techniques and Kadec's techniques (see for instance Section 10 in \cite[Chapter~1]{Young}), taking into account properties of the discrete Hilbert transforms as well as orthogonality results. 

\subsection{Theorem \ref{eq improvement decay}, optimal decay rates for interpolating functions and maximal perturbations} In Theorem \ref{eq improvement decay}, we have improved the uniform bound obtained by 
Radchenko and Viazovska \cite{RV} and, more recently, the sharper uniform bound by Bondarenko, Radchenko and Seip \cite{BRS} on the interpolating functions $a_n$ to one that \emph{decays} with $x;$ namely, we have that
\[
|a_n(x)| \lesssim n^{1/4}\log^{3/2}(1+n) \left(e^{-c|x|^2/n} 1_{|x| < Cn} + e^{-c|x|} 1_{|x|>Cn}\right),
\]
holds for all $n \in \N,$ where $C,c>0$ are two fixed positive constants. Although this improves the decay rates from before, the power $n^{1/4}$ found here and in \cite{BRS} in the growth seems likely not to be optimal; to that regard, we pose the following: 

\begin{question}\label{quest opti} What is the best decay rate for $a_n$ as in Theorem \ref{eq improvement decay}? Can one prove that $\sup_{x \in \R}|a_n(x)| = \mathcal{O}(1)$ in $n$? 
\end{question}

This conjectured growth seems to be the best possible, due to the recent findings of Bondarenko--Radchenko--Seip \cite{BRS}, which show that, for each $N \gg 1,$ the average
\[
\frac{1}{N+1} \sum_{ k \le N} |a_k(x)|^2 
\]
grows slower than some power of $\log{N}.$ 

Notice that, by a simple modification of the computations made in \S \ref{sec main result}, an affirmative answer to Question \ref{quest opti} yields an immediate improvement in the range of $\eps_i$ that we
allow for the theorems in \ref{sec main result}. Indeed, we get automatically that $|\eps_i| \lesssim i^{-1}$ is allowed in such results. On the other hand, this seems to be the best possible result one can 
achieve with our current methods, as the mean value theorem implies that $\sup_{x \in \R} |a_n'(x)| \gtrsim \sqrt{n}.$ 

In particular, all indicates that one needs a new idea in order to prove the following conjecture:

\begin{conjecture}[Maximal perturbations]\label{conj max pert} Let $f \in \mathcal{S}_{even}(\R)$ be a real function. Then there is $\theta >0$ so that, if $|\eps_i| < \theta, \, \forall \, i \in \N,$ 
$f$ can be uniquely recovered from its values 
\[
f(0),f(\sqrt{1+\eps_1}),f(\sqrt{2+\eps_2}),\dots,
\]
together with the values of its Fourier transform 
\[
\widehat{f}(0), \widehat{f}(\sqrt{1+\eps_1}), \widehat{f}(\sqrt{2+\eps_2}), \dots.
\]
\end{conjecture}
It might not be an easy task to prove Conjecture \ref{conj max pert} even with a new idea starting from our techniques, but we believe that the following version stands a chance of being more tractable with 
the current methods:
\begin{conjecture}[Maximal perturbations, weak form]\label{conj max pert weak} Let $f \in \mathcal{S}_{even}(\R)$ be a real function. Then, for each $a>0$, there is $\delta >0$ so that, if $|\eps_i| \le \delta k^{-a},$ 
then $f$ can be uniquely recovered from its values 
\[
f(0),f(\sqrt{1+\eps_1}),f(\sqrt{2+\eps_2}),\dots,
\]
together with the values of its Fourier transform 
\[
\widehat{f}(0), \widehat{f}(\sqrt{1+\eps_1}), \widehat{f}(\sqrt{2+\eps_2}), \dots.
\]
\end{conjecture}
In this framework, the results in \S \ref{sec main result} may be regarded as partial progress towards this conjecture. Notice that, by the remarks of \S \ref{sec unique small}, both versions of the conjecture 
imply that for each $\alpha  \in (0,1/2),$ there is $c_{\alpha}> 0$ so that if an even, real Schwartz function $f$ satisfies that $f(c_1 n^{\alpha}) = \widehat{f}(c_2 n^{\beta}) = 0$ and $c_1 < c_{\alpha}, \, c_2 < c_{\beta},$
then $f \equiv 0.$ These results can be compared, for instance, with our previous results in \cite{RS1}. 

\section*{Acknowledgements} 

We would like to thank Danylo Radchenko for several comments and suggestions in both early and later stages of development of this manuscript. We would also like to thank Felipe Gon\c calves for 
helpful discussions that led to the development of \S \ref{sec origin}. Finally, J.P.G.R. acknowledges financial support from CNPq.


\begin{thebibliography}{99}

\bibitem{ALV16}
\textsc{A Avantaggiati, P. Loretti, and P. Vellucci,}
\newblock \emph{Kadec-$1/4$ Theorem for sinc bases}, preprint at \href{https://arxiv.org/abs/1603.08762}{\texttt{arXiv:1603.08762}}.

\bibitem{BN}
\textsc{B. C. Berndt and M. I. Knopp,} 
\newblock \emph{Hecke’s Theory of Modular Forms and Dirichlet Series,}
\newblock World Scientific (2008).

\bibitem{BRS}
\textsc{A. Bondarenko, D. Radchenko and K. Seip,}
\newblock \emph{Fourier Interpolation with zeros of Zeta and $L-$functions,}
\newblock preprint at \href{https://arxiv.org/abs/2005.02996}{\texttt{arXiv:2005.02996}}.

\bibitem{BCK}
\textsc{J. Bourgain, L. Clozel, and J.-P. Kahane}, 
\newblock {\it Principe d'Heisenberg et fonctions positives.}
\newblock Ann. Inst. Fourier (Grenoble) {\bf 60} (2010), no.~4, 1215--1232.

\bibitem{Brezis}
\textsc{H. Brezis,}
\newblock \emph{Functional Analysis, Sobolev Spaces and Partial Differential Equations.}
\newblock Universitext, Springer--Verlag, 2011. 

\bibitem{Chandrasekharan} 
\textsc{K. Chandrasekharan}
\newblock \emph{Elliptic Functions}, 
\newblock Grundlehren der mathematischen Wissenschaften {\bf 281}, Springer-Verlag, 1985. 

\bibitem{CCK}
\textsc{J. Chung, S.-Y. Chung and D. Kim,}
\newblock \emph{Characterizations of the Gelfand-Shilov spaces via Fourier transforms},
\newblock Proc. Amer. Math. Soc. \textbf{124} (1996), no. 7, 2101--2108.

\bibitem{CohnGoncalves}
\textsc{H. Cohn and F. Gon\c calves},
\newblock \emph{An optimal uncertainty principle in twelve dimensions via modular forms},
\newblock Invent. Math. \textbf{217} (2019), no. 3, 799--831.

\bibitem{CKMRV1}
\textsc{H. Cohn, A. Kumar, S. Miller, D. Radchenko and M. Viazovska,}
\newblock \emph{The sphere packing problem in dimension 24,}  Ann. Math. {\bf 185} (2017), n. 3, 1017--1033.

\bibitem{CKMRV2}
\textsc{H. Cohn, A. Kumar, S. Miller, D. Radchenko and M. Viazovska,} 
\newblock \emph{Universal optimality of the $E_8$ and Leech lattices and interpolation formulas},
\newblock preprint at \href{https://arxiv.org/abs/1902.05438}{\texttt{arXiv:1902.05438}}. 

\bibitem{Goncalves}
\textsc{F. Gon\c{c}alves,}
\newblock \emph{Interpolation formulas with derivatives in de Branges spaces},
\newblock Trans. Amer. Math. Soc. {\bf 369} (2017), 805--832.

\bibitem{FelipeLittman} 
\textsc{F. Gon\c calves and F. Littmann,}
\newblock \emph{Interpolation formulas with derivatives in de Branges spaces II},
\newblock J. Math. Anal. Appl. {\bf 458} (2018), n. 2, 1091--1114.

\bibitem{GOSR20}
\textsc{F. Gon\c{c}alves, D. Oliveira e Silva, and J. P. G. Ramos}, 
\newblock {\it On regularity and mass concentration phenomena for the sign uncertainty principle},
\newblock Preprint at \href{https://arxiv.org/abs/2003.10765}{\texttt{arXiv:2003.10765}}.

\bibitem{GOSR201}
\textsc{F. Gon\c{c}alves, D. Oliveira e Silva, and J. P. G. Ramos},
\newblock \emph{New sign uncertainty principles},
\newblock Preprint at  \href{https://arxiv.org/abs/2003.10771}{\texttt{arXiv:2003.10771}}. 

\bibitem{GOSS17}
\textsc{F. Gon\c{c}alves, D. Oliveira e Silva, and S. Steinerberger}, 
\newblock {\it Hermite polynomials, linear flows on the torus, and an uncertainty principle for roots},
\newblock J. Math. Anal. Appl. {\bf 451} (2017), n.~2, 678--711. 

\bibitem{K14}
\textsc{M.I. Kadec,}
\newblock \emph{The exact value of the Paley-Wiener constant}, Sov. Math. Dokl. 5, 1964, 559--561. 

\bibitem{KurasovSarnak}
\textsc{P. Kurasov and P. Sarnak,}
\newblock \emph{Stable polynomials and crystalline measures},
\newblock Preprint at  \href{https://arxiv.org/abs/2004.05678}{\texttt{arXiv:2004.05678}}

\bibitem{LO13}
\textsc{N. Lev and A. Olevskii,}
\newblock \emph{Measures with uniformly discrete support and spectrum},
C. R. Math. Acad. Sci. Paris \textbf{351} (2013), no. 15-16, 613-617.

\bibitem{LO15}
\textsc{N. Lev and A. Olevskii,}
\newblock \emph{Quasicrystal and Poisson's summation formula},
Invent. Math. \textbf{200} (2015), no. 2, 585-606.

\bibitem{Lit02} 
\textsc{F. Littmann,}
\emph{Entire majorants via Euler-Maclaurin summation.}

\bibitem{LyuSeip}
\textsc{Y. Lyubarskii and K. Seip,}
\emph{Weighted Paley-Wiener spaces},
\newblock J. Amer. Math. Soc. {\bf 15} (2002), n.~4, 979--1006. 

\bibitem{Meyer}
\textsc{Y. Meyer,} 
\newblock \emph{Measures with locally finite support and spectrum}, 
\newblock Rev. Mat. Iberoam. \textbf{33} (2017), no. 3, 1025--1036.

\bibitem{PW} 
\textsc{R. Paley and N. Wiener,}
\emph{Fourier transforms in the complex domain}, Amer. Math. Soc. Colloquium Publications vol. 19. Amer. Math. Soc., New York, 1934

\bibitem{RV}
\textsc{D. Radchenko and M. Viazovska,}
\newblock \emph{Fourier interpolation on the real line},
\newblock Publ. Math. Inst. Hautes Études Sci. \textbf{129} (2019), 51–81.

\bibitem{RS1}
\textsc{J.P.G. Ramos and M. Sousa,}
\newblock \emph{Fourier uniqueness pairs of powers of integers},
\newblock arXiv preprint at \href{https://arxiv.org/abs/1910.04276}{\texttt{arXiv:1910.04276}}. 

\bibitem{Schechter} 
\textsc{M. Schechter,}
\newblock \emph{Basic theory of Fredholm operators},
\newblock Ann. Scuola Norm. Pisa, Classe di Scienze $3^e$ s\'erie, \textbf{21} (1967), n. 2, p. 261--280.

\bibitem{SeipOC}
\textsc{K. Seip and J. Ortega-Cerd\`a},
\newblock \emph{Fourier frames},
\newblock Ann. Math. {\bf 155} (2002), 789--806.

\bibitem{Shannon} 
\textsc{C. E. Shannon,}
\newblock \emph{Communications in the presence of noise},
\newblock Proc. IRE {\bf 37} (1949), 10--21. 

\bibitem{Stoller}
\textsc{M. Stoller,}
\newblock \emph{Fourier interpolation from spheres},
\newblock Preprint at  \href{https://arxiv.org/abs/2002.11627}{\texttt{arXiv:2002.11627}}. 

\bibitem{Vaaler} 
\textsc{J. D. Vaaler,}
\newblock \emph{ Some extremal functions in Fourier analysis},  Bull. Amer.Math. Soc. {\bf 12} (1985), 183--215.

\bibitem{Viazovska1}
\textsc{M. Viazovska,}
\newblock \emph{The sphere packing problem in dimension 8}, Ann. Math. {\bf 185} (2017), n. 3, 991--1015. 

\bibitem{Whittaker}
\textsc{E. T. Whittaker},
\newblock \emph{On the functions which are represented by the expansions of the interpolation theory},
\newblock Proc. Royal Soc. Edinburgh. {\bf 35} (1915), 181--194.

\bibitem{Young}
\textsc{R. M. Young},
\newblock \emph{An introduction to nonharmonic Fourier series},
\newblock Academic Press, 1980.

\bibitem{Zagier} 
\textsc{D. Zagier,}
\newblock \emph{Elliptic modular forms and their applications,}
\newblock in \emph{The 1-2-3 of Modular Forms (K. Ranestad, ed.)}, 1--103, Universitext, Springer, Berlin (2008).

\end{thebibliography}
\end{document}